\documentclass[12pt]{article}

\usepackage[margin=1in]{geometry}
\usepackage{amsmath,amssymb,amsthm}
\usepackage{bm}
\usepackage{booktabs}
\usepackage{graphicx}
\usepackage{float}
\usepackage{cleveref}
\usepackage{xcolor}

\theoremstyle{definition}

\newtheorem{prop}{Proposition}
\newtheorem{coro}{Corollary}
\newtheorem{theorem}{Theorem}
\newtheorem{observation}{Observation}
\crefname{prop}{proposition}{propositions}
\Crefname{prop}{Proposition}{Propositions}

\newcommand{\R}{\mathbb{R}}
\newcommand{\C}{\mathbb{C}}
\newcommand{\V}[1]{\bm{#1}}
\newcommand{\cE}{\mathcal{E}}
\newcommand{\cI}{\mathcal{I}}

\begin{document}

\title{Compact and Infinite-Order Error Analysis\\
for Null-Space SVD Estimation}
\author{Xin Li, Jonathan Cohen, Rami Puzis}

\maketitle

\begin{abstract}
We study null-space estimation from a noisy matrix.  For a simple left null
space, we first derive an exact compact expression for the error of the
smallest left singular vector.  We then give an all-order series for the SVD
vector and projector, followed by compact and consistently truncated series
forms for the fixed-realization empirical risk and conditional population
generalization risk.  The recursion extends to a multiple-dimensional null
space by following the complete invariant subspace.  The convergence radius
is not inferred from an error plot: it is computed independently from the
nearest complex exceptional point that joins a retained eigenvalue branch to
its complement.  A reduced-nullity experiment shows that moving this
spectral boundary can increase the radius, although the improvement is not
monotone in the retained nullity.  For individually ordered null directions
under Gaussian training with \(\tau\geq m\), we prove that the Wishart
splitting matrix \(W\) gives a strict second-order empirical ranking.  Gaussian averaging
equalizes the leading generalization
risks at both small and very large noise, while a column-swap theorem proves
strict expected generalization ranking for an isotropic signal subspace.  For
unequal spikes, an exact population-overlap criterion and a simultaneous
\(99\%\) Monte Carlo confidence certificate explain the observed intermediate
ranking.  A sixth-order risk correction improves the lower-crossover estimate
in the reported experiment.  This
equal--ranked--equal phenomenon is a
finite-sample diagnostic related to spectral mixing, but its tolerance
crossings, the exceptional-point radius, and the asymptotic BBP threshold are
three distinct quantities.
\end{abstract}

\section{Contributions}

The main contributions are as follows.
\begin{enumerate}
    \item We give an exact, compact
    reconstruction of the smallest left singular vector for simple nullity.
    We then derive scalar, block, and basis-free projector recursions to
    arbitrary order, together with exact fixed-realization compact and
    coefficient forms for the empirical and population generalization risks.

    \item We explain the different rank behavior of the two risks.  Empirical branch risk is exactly the corresponding ordered sample eigenvalue divided by the number of measurements \(\tau\). 
    Its first nonzero small-noise splitting is governed by the effective matrix \(W\).
    Under Gaussian training with \(\tau\geq m\),
    this matrix is a nonsingular Wishart and gives a strict expected second-order ranking.
    For generalization risk, the stated \(L^1\) remainder conditions give common small- and large-noise endpoint laws.
    Under Gaussian training with \(\tau\geq m\), we prove exact strict ranking for an isotropic signal subspace.
    For unequal signal eigenvalues, we give an exact cumulative population-overlap criterion.

    \item Under generic boundary collisions, we characterize the convergence
    disk of the projector series by the nearest complex exceptional point
    joining a retained eigenvalue sheet to its complement.  We also identify
    the possible cancellation condition for a scalar generalization risk and
    show why changing the retained nullity can enlarge or shrink the radius:
    it moves the spectral boundary and can expose a different collision.

    \item Numerical experiments verify the compact identities and convergence
    of the arbitrary-order series.  They also verify the exceptional-point radii and the nonmonotonic effect of reducing nullity.
    For the reported unequal-spike example, a simultaneous \(99\%\) Monte Carlo confidence certificate verifies the overlap inequalities and expected generalization ranking at a representative intermediate-noise point.  Calibrated fourth- and sixth-order terms accurately predict the lower visibility crossing, while a calibrated inverse-noise coefficient predicts the upper crossing.
\end{enumerate}

\section{Problem setting}

Let
\begin{equation}
    Z=HX\in\R^{m\times \tau},
    \qquad
    E(\sigma)=\sigma\cE,
    \qquad
    \widetilde Z(\sigma)=Z+E(\sigma),
    \label{eq:data-model}
\end{equation}
where \(H\in\R^{m\times n}\), \(X\in\R^{n\times\tau}\), and
\(\cE\in\R^{m\times\tau}\) is a fixed standardized unit-variance noise realization.
Set
\begin{equation}
    r=\operatorname{rank}(Z),
    \qquad q_0=m-r,
\end{equation}
and let \(Q_0\in\R^{m\times q_0}\) have orthonormal columns spanning
\(\operatorname{Null}(Z^T)\).  If \(X\) has full row rank and
\(\operatorname{rank}(H)=n\), then \(r=n\), \(q_0=m-n\), and the
measurement-based and system-based left null spaces coincide.

For \(1\leq q\leq q_0\), the noisy \(q\)-dimensional estimator is
\begin{equation}
    \widehat Q_q(\sigma)
    \in\arg\min_{U^TU=I_q}
        \|\widetilde Z(\sigma)^TU\|_F^2.
    \label{eq:subspace-estimator}
\end{equation}
The columns of $\widehat Q_q(\sigma)$ are the \(q\) smallest left singular vectors of \(\widetilde Z\).
Because a basis may rotate without changing its subspace,
the invariant object is the orthogonal projector
\begin{equation}
    \widehat P_q(\sigma)=\widehat Q_q(\sigma)\widehat Q_q(\sigma)^T.
    \label{eq:noisy-projector}
\end{equation}
The choice \(q=q_0\) estimates the complete clean null space.  Choosing
\(q<q_0\) will be called \emph{reducing the retained nullity}.

For simple-nullity case \(q_0=1\), write \(Q_0=\V{\eta}\), where
\begin{equation}
    Z^T\V{\eta}=\V{0},
    \qquad \|\V{\eta}\|_2=1.
\end{equation}
After selecting the sign so that
\(\V{\eta}^T\widehat{\V{\eta}}>0\), define
\begin{equation}
    \V{\epsilon}(\sigma)
      =\widehat{\V{\eta}}(\sigma)-\V{\eta}.
    \label{eq:error-definition}
\end{equation}
Throughout, bold \(\V{\epsilon}\) denotes a vector estimation error.  
We use \(R_{\mathrm{emp},q}\) and \(R_{\mathrm{gen},q}\) for scalar risks conditional on one finite training realization, and
\(\varepsilon_{\mathrm{emp},q}\) and \(\varepsilon_{\mathrm{gen},q}\) for
the corresponding expectations over training realizations.

\section{Exact eigenvalue problem and compact solution}

The SVD estimator is equivalent to a symmetric eigenvalue problem.  
Define
\begin{align}
    A(\sigma)
      &=\widetilde Z(\sigma)\widetilde Z(\sigma)^T
        =A_0+\sigma A_1+\sigma^2A_2,
        \label{eq:gram-expansion}\\
    A_0&=ZZ^T,
    &A_1&=Z\cE^T+\cE Z^T,
    &A_2&=\cE\cE^T.
    \label{eq:coefficient-matrices}
\end{align}
For the simple-nullity case \(q_0=1\), the exact estimator satisfies
\begin{equation}
    A(\sigma)\widehat{\V{\eta}}(\sigma)
      =\lambda(\sigma)\widehat{\V{\eta}}(\sigma),
    \qquad
    \lambda(\sigma)
      =\|\widetilde Z(\sigma)^T\widehat{\V{\eta}}(\sigma)\|_2^2.
    \label{eq:exact-eigenproblem}
\end{equation}

\subsection{Compact-form solution for a simple null space}

In this subsection \(q_0=1\).
For the noise matrix \(E=E(\sigma)=\sigma\cE\), define
\begin{align}
    G&=(ZZ^T)^+,
    &K&=GZE^T=(Z^+)^TE^T,
    \label{eq:compact-KG}\\
    L&=GEZ^T,
    &M&=GEE^T,
    \label{eq:compact-LM}\\
    \mathcal{Q}&=I_m+K+L+M-\lambda G.
    \label{eq:compact-Q}
\end{align}

\begin{prop}[Exact compact representation and scalar closure]
\label{prop:compact}
Let \(\lambda\in\R\) be such that \(\mathcal Q=\mathcal Q(\sigma,\lambda)\)
is invertible, and define
\begin{equation}
    \V{w}=\mathcal{Q}^{-1}\V{\eta},
    \qquad
    \V{v}=\mathcal{Q}^{-1}(K+M)\V{\eta}.
    \label{eq:compact-w-v}
\end{equation}
Then
\begin{equation}
    \V{\eta}^T\V{w}=1,
    \qquad
    \V{v}=\V{\eta}-\V{w}.
    \label{eq:compact-w-v-identities}
\end{equation}
Moreover, \(\lambda\) is an eigenvalue of
\(A=(Z+E)(Z+E)^T\) if and only if it satisfies the scalar closure
\begin{equation}
    \boxed{
    F(\sigma,\lambda)
      :=\lambda-\V{\eta}^TE(Z+E)^T\V{w}=0.}
    \label{eq:compact-secular-condition}
\end{equation}
For any such root, the two unit eigenvectors on this eigenline and their
errors are
\begin{equation}
    \widehat{\V{\eta}}_{\pm}
      =\mathbin{\pm}\frac{\V{w}}{\|\V{w}\|_2},
    \qquad
    \V{\epsilon}_{\pm}
      =\mathbin{\pm}\frac{\V{w}}{\|\V{w}\|_2}-\V{\eta}.
    \label{eq:compact-two-signs}
\end{equation}
The orientation convention
\(\V{\eta}^T\widehat{\V{\eta}}>0\) selects the plus branch, so
\begin{equation}
    \boxed{\V{\epsilon}=\frac{\V{w}}{\|\V{w}\|_2}-\V{\eta}.}
    \label{eq:compact-oriented-solution}
\end{equation}

Equivalently, let
\begin{equation}
    \alpha_{\pm}=-1\mathbin{\pm}\frac{1}{\|\V{w}\|_2}.
    \label{eq:compact-alpha-simple}
\end{equation}
Then
\begin{equation}
    \boxed{
    \V{\epsilon}_{\pm}
      =\mathcal{Q}^{-1}\bigl[\alpha_{\pm}I_m-K-M\bigr]\V{\eta}
      =\alpha_{\pm}\V{w}-\V{v}.}
    \label{eq:compact-error}
\end{equation}
Normalization also gives the following unsimplified quadratic formula for
\(\alpha\).  It is equivalent to \eqref{eq:compact-alpha-simple}:
\begin{equation}
    \boxed{
    \alpha=
    \frac{
      \V{w}^T\V{v}-\V{\eta}^T\V{w}
      \mathbin{\pm}
      \sqrt{
        (\V{\eta}^T\V{w}-\V{w}^T\V{v})^2
        -\|\V{w}\|_2^2
          (\|\V{v}\|_2^2-2\V{\eta}^T\V{v})
      }
    }{\|\V{w}\|_2^2}.}
    \label{eq:compact-alpha}
\end{equation}

Writing
\(\beta=\V{\eta}^T\widehat{\V{\eta}}\), the oriented eigenpair is
equivalently characterized by the closed compact system
\begin{equation}
 \boxed{
 \mathcal Q\widehat{\V{\eta}}=\beta\V{\eta},
 \qquad
 \lambda\beta=\V{\eta}^TE(Z+E)^T\widehat{\V{\eta}},
 \qquad
 \widehat{\V{\eta}}^T\widehat{\V{\eta}}=1,
 \qquad \beta>0.}
 \label{eq:closed-compact-system}
\end{equation}
Along \(E(\sigma)=\sigma\cE\), there is a unique analytic root of
\eqref{eq:compact-secular-condition} near \((\sigma,\lambda)=(0,0)\)
with \(\lambda(0)=0\).  It is the eigenvalue branch descending from the
clean zero eigenvalue.
\end{prop}

\begin{proof}
Simple nullity gives
\begin{equation}
  GZZ^T=(ZZ^T)^+ZZ^T
    =P_{\operatorname{Range}(Z)}
    =I_m-\V{\eta}\V{\eta}^T,
  \qquad
  G\V{\eta}=\V{0}.
  \label{eq:compact-pseudoinverse-identities}
\end{equation}
Expanding \(A=(Z+E)(Z+E)^T\) and using the definitions of \(K,L,M\)
shows directly where \(\mathcal Q\) comes from:
\begin{align}
 G(A-\lambda I_m)
 &=GZZ^T+GZE^T+GEZ^T+GEE^T-\lambda G \notag\\
 &=\mathcal Q-\V{\eta}\V{\eta}^T.
 \label{eq:compact-residual-identity}
\end{align}
Because \(G\) is symmetric and \(G\V{\eta}=\V{0}\), every nonidentity
term in \(\mathcal Q\) has zero left product with \(\V{\eta}^T\).  Hence
\begin{equation}
 \V{\eta}^T\mathcal Q=\V{\eta}^T,
 \qquad
 \V{\eta}^T\V{w}
   =\V{\eta}^T\mathcal Q^{-1}\V{\eta}=1.
 \label{eq:compact-w-normalization}
\end{equation}
Also \(Z^T\V{\eta}=\V{0}\), so
\(L\V{\eta}=G\V{0}=\V{0}\) and \(G\V{\eta}=\V{0}\).  Therefore
\begin{equation}
 \mathcal Q\V{\eta}=\V{\eta}+(K+M)\V{\eta},
\end{equation}
and multiplication by \(\mathcal Q^{-1}\) proves
\(\V{v}=\V{\eta}-\V{w}\).

To prove the scalar closure, put
\(\V{r}=(A-\lambda I_m)\V{w}\).  Since
\(\mathcal Q\V{w}=\V{\eta}\),
\eqref{eq:compact-residual-identity} and
\eqref{eq:compact-w-normalization} give
\begin{equation}
 G\V{r}
 =\mathcal Q\V{w}
  -\V{\eta}\V{\eta}^T\V{w}
 =\V{0}.
\end{equation}
Thus \(\V{r}\in\operatorname{Null}(G)=\operatorname{span}\{\V{\eta}\}\).
Furthermore, \(\V{\eta}^TZZ^T=\V{0}^T\) and
\(\V{\eta}^TZE^T=\V{0}^T\), so
\begin{equation}
 \V{\eta}^T\V{r}
 =\V{\eta}^TE(Z+E)^T\V{w}-\lambda
 =-F(\sigma,\lambda).
\end{equation}
Consequently, \(F(\sigma,\lambda)=0\) if and only if \(\V{r}=\V{0}\),
so every root of the scalar closure is an eigenvalue.  Conversely, suppose
\((A-\lambda I_m)\V{u}=\V{0}\) for some \(\V{u}\ne\V{0}\).  Applying
\eqref{eq:compact-residual-identity} gives
\begin{equation}
 \mathcal Q\V{u}
   =(\V{\eta}^T\V{u})\V{\eta}.
\end{equation}
If \(\V{\eta}^T\V{u}=0\), invertibility of \(\mathcal Q\) would imply
\(\V{u}=\V{0}\), a contradiction.  Hence
\(\gamma:=\V{\eta}^T\V{u}\ne0\), and
\begin{equation}
 \V{u}=\gamma\mathcal Q^{-1}\V{\eta}=\gamma\V{w}.
\end{equation}
Thus \(\V{w}\) is an eigenvector and the preceding calculation forces
\(F(\sigma,\lambda)=0\).  It also shows that the eigenspace is the single
line generated by \(\V{w}\).  This proves the equivalence in
\eqref{eq:compact-secular-condition}; normalization gives
\eqref{eq:compact-two-signs}, and
\eqref{eq:compact-w-normalization} selects the plus sign.

Using \(\V{v}=\V{\eta}-\V{w}\) and
\(\alpha_{\pm}=-1\pm\|\V{w}\|_2^{-1}\), we obtain
\begin{equation*}
 \alpha_{\pm}\V{w}-\V{v}
 =\mathbin{\pm}\frac{\V{w}}{\|\V{w}\|_2}-\V{\eta},
\end{equation*}
which proves \eqref{eq:compact-error}.  Alternatively, substituting
\(\V{\epsilon}=\alpha\V{w}-\V{v}\) into
\(\|\V{\eta}+\V{\epsilon}\|_2^2=1\) gives
\begin{equation}
 \|\V{w}\|_2^2\alpha^2
 +2(\V{\eta}^T\V{w}-\V{w}^T\V{v})\alpha
 +\|\V{v}\|_2^2-2\V{\eta}^T\V{v}=0.
\end{equation}
The quadratic formula is \eqref{eq:compact-alpha}; applying
\eqref{eq:compact-w-v-identities} reduces its roots to
\eqref{eq:compact-alpha-simple}.

For the oriented eigenvector, set
\(\beta=\V{\eta}^T\widehat{\V{\eta}}=\|\V{w}\|_2^{-1}\).
Then \(\widehat{\V{\eta}}=\beta\V{w}\), so
\(\mathcal Q\widehat{\V{\eta}}=\beta\V{\eta}\).  Multiplying
\eqref{eq:compact-secular-condition} by \(\beta\) gives the second identity
in \eqref{eq:closed-compact-system}; the third is normalization.
Conversely, the first equation of \eqref{eq:closed-compact-system} and
invertibility of \(\mathcal Q\) give
\(\widehat{\V{\eta}}=\beta\V{w}\).  Its normalization and \(\beta>0\)
give \(\beta=\|\V{w}\|_2^{-1}\), while the second equation is
\(\beta F(\sigma,\lambda)=0\).  Thus \(F(\sigma,\lambda)=0\), proving
the claimed equivalence of the closed system and the oriented eigenpair.

Finally, for \(E(\sigma)=\sigma\cE\), the function
\begin{equation}
 F(\sigma,\lambda)
 =\lambda-\V{\eta}^T\sigma\cE
       (Z+\sigma\cE)^T
       \mathcal Q(\sigma,\lambda)^{-1}\V{\eta}
\end{equation}
is analytic near \((0,0)\), because \(\mathcal Q(0,0)=I_m\).  It satisfies
\(F(0,0)=0\) and \(\partial_\lambda F(0,0)=1\), since the second term has
an explicit factor \(\sigma\).  The analytic implicit-function theorem
therefore gives the unique local analytic root \(\lambda(\sigma)\) with
\(\lambda(0)=0\).  The preceding equivalence identifies it with the desired
eigenvalue branch.
\end{proof}

\section{Infinite-order recursion}

\subsection{Simple-nullity recursion}

Continue to assume \(q_0=1\), and expand
\begin{align}
    \widehat{\V{\eta}}(\sigma)
      &=\sum_{k=0}^{\infty}\sigma^k\V{\eta}_k,
      &\V{\eta}_0&=\V{\eta},
      &\V{\eta}_{\ell}&=\V{0}\quad(\ell<0),
      \label{eq:eigenvector-series}\\
    \lambda(\sigma)
      &=\sum_{k=0}^{\infty}\sigma^k\lambda_k,
      &\lambda_0&=0.
      \label{eq:eigenvalue-series}
\end{align}
Equivalently, these are the Taylor coefficients
\begin{equation}
    \V{\eta}_k
      =\frac{1}{k!}\frac{d^k\widehat{\V{\eta}}}{d\sigma^k}(0),
    \qquad
    \lambda_k
      =\frac{1}{k!}\frac{d^k\lambda}{d\sigma^k}(0).
    \label{eq:Taylor-coefficient-definition}
\end{equation}

The next result is not an independent construction: it is the coefficientwise
Taylor expansion of the exact closed compact system in
\Cref{prop:compact}.

\begin{prop}[Taylor expansion of the compact solution]
\label{prop:scalar-recursion}
For \(k\geq1\), write
\begin{equation}
    c_k:=\V{\eta}^T\V{\eta}_k,
    \qquad
    \V{\eta}_k^\perp
      :=(I_m-\V{\eta}\V{\eta}^T)\V{\eta}_k.
    \label{eq:coefficient-components}
\end{equation}
With empty sums interpreted as zero, the Taylor coefficients are determined
sequentially, for every \(k\geq1\), by
\begin{align}
    \lambda_k
      &=\V{\eta}^T A_1\V{\eta}_{k-1}
        +\V{\eta}^T A_2\V{\eta}_{k-2}
        -\sum_{j=1}^{k-1}\lambda_j
             \V{\eta}^T\V{\eta}_{k-j},
      \label{eq:lambda-recursion}\\
    \V{\eta}_k^{\perp}
      &=G\left(
          -A_1\V{\eta}_{k-1}
          -A_2\V{\eta}_{k-2}
          +\sum_{j=1}^{k-1}\lambda_j\V{\eta}_{k-j}
        \right),
      \label{eq:perpendicular-recursion}\\
    c_k
      &=-\frac12\sum_{j=1}^{k-1}
          \V{\eta}_j^T\V{\eta}_{k-j},
      \label{eq:parallel-recursion}\\
    \V{\eta}_k
      &=\V{\eta}_k^{\perp}+c_k\V{\eta}.
      \label{eq:combine-recursion}
\end{align}
The order-\(K\) Taylor approximation to the error is
\begin{equation}
    \V{\epsilon}^{[K]}(\sigma)
       =\sum_{k=1}^{K}\sigma^k\V{\eta}_k.
    \label{eq:order-K-solution}
\end{equation}
\end{prop}
\begin{proof}
The local analytic root and its oriented eigenvector are supplied by
\Cref{prop:compact}.  Define
\begin{equation}
 \beta(\sigma)
   =\V{\eta}^T\widehat{\V{\eta}}(\sigma)
   =\sum_{k=0}^{\infty}\sigma^kc_k,
 \qquad
 c_0=1,
 \qquad
 c_k=\V{\eta}^T\V{\eta}_k\quad(k\geq1).
 \label{eq:beta-series}
\end{equation}
Since \(E(\sigma)=\sigma\cE\), the definitions of \(K,L,M\) and
\eqref{eq:coefficient-matrices} give
\begin{equation}
 K+L=\sigma GA_1,
 \qquad
 M=\sigma^2GA_2,
 \qquad
 \mathcal Q
  =I_m+\sigma GA_1+\sigma^2GA_2-\lambda(\sigma)G.
 \label{eq:Q-scaled-series-form}
\end{equation}
Moreover, \(\V{\eta}^TZ=\V{0}^T\), and hence
\begin{equation}
 \V{\eta}^TE(Z+E)^T
 =\V{\eta}^T(\sigma A_1+\sigma^2A_2).
\end{equation}
Therefore, the scalar equation in the closed compact system
\eqref{eq:closed-compact-system} becomes
\begin{equation}
 \lambda(\sigma)\beta(\sigma)
 =\V{\eta}^T(\sigma A_1+\sigma^2A_2)
       \widehat{\V{\eta}}(\sigma).
 \label{eq:compact-scalar-generating-equation}
\end{equation}

We first expand this scalar identity.  Its left-hand side is the Cauchy
product
\begin{align*}
 \lambda(\sigma)\beta(\sigma)
 &=\left(\sum_{j=0}^{\infty}\sigma^j\lambda_j\right)
   \left(\sum_{\ell=0}^{\infty}\sigma^\ell c_\ell\right)\\
 &=\sum_{k=0}^{\infty}\sigma^k
      \sum_{j=0}^{k}\lambda_jc_{k-j},
\end{align*}
because the exponents contributing to order \(k\) satisfy
\(j+\ell=k\).  Its right-hand side is
\begin{equation*}
 \sum_{k=0}^{\infty}\sigma^k
 \left(
   \V{\eta}^TA_1\V{\eta}_{k-1}
  +\V{\eta}^TA_2\V{\eta}_{k-2}
 \right),
\end{equation*}
where \(\V{\eta}_\ell=\V{0}\) for \(\ell<0\).  Equality of Taylor
coefficients, together with \(\lambda_0=0\) and \(c_0=1\), gives
\begin{equation}
 \lambda_k+\sum_{j=1}^{k-1}\lambda_jc_{k-j}
 =\V{\eta}^TA_1\V{\eta}_{k-1}
  +\V{\eta}^TA_2\V{\eta}_{k-2},
 \qquad k\geq1.
 \label{eq:compact-lambda-coefficient}
\end{equation}
Substituting \(c_{k-j}=\V{\eta}^T\V{\eta}_{k-j}\) and isolating
\(\lambda_k\) proves \eqref{eq:lambda-recursion}.  Thus the eigenvalue
recursion is the Taylor expansion of the scalar closure in
\Cref{prop:compact}.

Next expand the vector equation
\(\mathcal Q\widehat{\V{\eta}}=\beta\V{\eta}\) from
\eqref{eq:closed-compact-system}.  By
\eqref{eq:Q-scaled-series-form},
\begin{align*}
 \mathcal Q\widehat{\V{\eta}}
 &=\left(
   I_m+\sigma GA_1+\sigma^2GA_2
   -\left(\sum_{j=0}^{\infty}\sigma^j\lambda_j\right)G
   \right)
   \left(\sum_{\ell=0}^{\infty}
                   \sigma^\ell\V{\eta}_\ell\right)\\
 &=\sum_{k=0}^{\infty}\sigma^k\left(
   \V{\eta}_k+GA_1\V{\eta}_{k-1}+GA_2\V{\eta}_{k-2}
   -\sum_{j=0}^{k}\lambda_jG\V{\eta}_{k-j}
   \right).
\end{align*}
As in the scalar Cauchy product, the inner sum collects precisely the pairs
of exponents whose sum is \(k\).  On the other side,
\begin{equation*}
 \beta\V{\eta}
 =\sum_{k=0}^{\infty}\sigma^kc_k\V{\eta}.
\end{equation*}
Matching the coefficient of \(\sigma^k\) gives
\begin{equation}
 \V{\eta}_k+GA_1\V{\eta}_{k-1}+GA_2\V{\eta}_{k-2}
 -\sum_{j=0}^{k}\lambda_jG\V{\eta}_{k-j}
 =c_k\V{\eta}.
 \label{eq:compact-vector-coefficient}
\end{equation}
The \(j=0\) term vanishes because \(\lambda_0=0\), and the \(j=k\) term
vanishes because \(G\V{\eta}_0=G\V{\eta}=\V{0}\).  Rearranging yields
\begin{equation}
 \V{\eta}_k-c_k\V{\eta}
 =G\left(
   -A_1\V{\eta}_{k-1}-A_2\V{\eta}_{k-2}
   +\sum_{j=1}^{k-1}\lambda_j\V{\eta}_{k-j}
 \right).
 \label{eq:compact-perpendicular-coefficient}
\end{equation}
By \eqref{eq:beta-series}, the left-hand side is
\((I_m-\V{\eta}\V{\eta}^T)\V{\eta}_k=\V{\eta}_k^\perp\).  This proves
\eqref{eq:perpendicular-recursion}: it is the Taylor expansion of
the compact vector equation in \Cref{prop:compact}.  This determines only the
perpendicular component. Normalization determines the component parallel to
\(\V{\eta}\).  Expanding the
normalization identity by another Cauchy product gives
\begin{align*}
 1
 &=\widehat{\V{\eta}}(\sigma)^T
   \widehat{\V{\eta}}(\sigma)\\
 &=\left(\sum_{r=0}^{\infty}\sigma^r\V{\eta}_r\right)^T
   \left(\sum_{\ell=0}^{\infty}
                \sigma^\ell\V{\eta}_\ell\right)\\
 &=\sum_{r=0}^{\infty}\sum_{\ell=0}^{\infty}
      \sigma^{r+\ell}\V{\eta}_r^T\V{\eta}_\ell\\
 &=\sum_{k=0}^{\infty}\sigma^k
      \sum_{j=0}^{k}\V{\eta}_j^T\V{\eta}_{k-j}.
\end{align*}
The coefficient at order zero is
\(\V{\eta}_0^T\V{\eta}_0=\|\V{\eta}\|_2^2=1\).  Since the right-hand
side of the normalization identity is the constant series \(1\), every
coefficient of order \(k\geq1\) must vanish.  Hence
\begin{equation}
 0=\sum_{j=0}^{k}\V{\eta}_j^T\V{\eta}_{k-j},
 \qquad k\geq1.
 \label{eq:normalization-coefficient}
\end{equation}
The endpoint terms \(j=0\) and \(j=k\) are
\begin{equation*}
 \V{\eta}_0^T\V{\eta}_k+
 \V{\eta}_k^T\V{\eta}_0
 =2\V{\eta}^T\V{\eta}_k,
\end{equation*}
where \(\V{\eta}_0=\V{\eta}\) and the two real scalar products are equal.
Separating these endpoint terms gives
\begin{equation}
 0=2\V{\eta}^T\V{\eta}_k
   +\sum_{j=1}^{k-1}\V{\eta}_j^T\V{\eta}_{k-j}.
 \label{eq:normalization-split}
\end{equation}
Defining \(c_k=\V{\eta}^T\V{\eta}_k\) and solving
\eqref{eq:normalization-split} for this scalar gives
\begin{equation*}
 c_k=-\frac12\sum_{j=1}^{k-1}
          \V{\eta}_j^T\V{\eta}_{k-j},
\end{equation*}
which is \eqref{eq:parallel-recursion}.  Both indices in each summand lie
between \(1\) and \(k-1\), so only previously computed coefficients occur.
In particular, the empty sum at \(k=1\) gives \(c_1=0\), while
\(c_2=-\frac12\|\V{\eta}_1\|_2^2\).

Finally, since \(\|\V{\eta}\|_2=1\), the complementary orthogonal
projectors are \(I_m-\V{\eta}\V{\eta}^T\) and
\(\V{\eta}\V{\eta}^T\).  Therefore
\begin{align}
 \V{\eta}_k
 &=(I_m-\V{\eta}\V{\eta}^T)\V{\eta}_k
   +\V{\eta}\V{\eta}^T\V{\eta}_k \notag\\
 &=\V{\eta}_k^\perp+c_k\V{\eta}.
 \label{eq:eta-k-decomposition}
\end{align}
The first term is the perpendicular component already determined by
\eqref{eq:perpendicular-recursion}; hence this is precisely
\eqref{eq:combine-recursion}.  The formulas are sequential: first compute
\(\lambda_k\), then \(\V{\eta}_k^\perp\), then \(c_k\), and finally combine
the two components to obtain \(\V{\eta}_k\).  Subtracting
\(\V{\eta}_0=\V{\eta}\) from the truncated eigenvector series gives
\eqref{eq:order-K-solution}.  Hence
\eqref{eq:lambda-recursion}--\eqref{eq:combine-recursion} are the Taylor
coefficient equations of the exact compact system in \Cref{prop:compact},
whereas \eqref{eq:order-K-solution} is their order-\(K\) truncation.
\end{proof}

The first coefficient is
\begin{equation}
    \lambda_1=0,
    \qquad
    \V{\eta}_1=-GA_1\V{\eta}=-GZ\cE^T\V{\eta}.
\end{equation}
Here \(\lambda_1=\V{\eta}^TA_1\V{\eta}=0\), because
\(Z^T\V{\eta}=\V{0}\) (and hence \(\V{\eta}^TZ=\V{0}^T\)).
Therefore
\begin{equation}
    \V{\epsilon}(\sigma)
      =-\sigma GZ\cE^T\V{\eta}+O(\sigma^2)
      =-(Z^T)^+E^T\V{\eta}+O(\sigma^2),
    \label{eq:first-order-solution}
\end{equation}
which is the first-order least-squares expression.
At second order,
\begin{align}
    \lambda_2
      &=\V{\eta}^TA_1\V{\eta}_1+
        \V{\eta}^TA_2\V{\eta},\\
    \V{\eta}_2
      &=-G(A_1\V{\eta}_1+A_2\V{\eta})
        -\frac12\|\V{\eta}_1\|_2^2\V{\eta}.
\end{align}
The final term is the normalization correction.

\subsection{Block recursion for multiple nullity}

When \(q_0>1\), the clean matrix determines a null subspace but not
preferred individual null vectors.  Therefore, the scalar recursion cannot
be applied independently to arbitrary columns.  Recall that
\(Q_0\in\R^{m\times q_0}\) is the fixed clean orthonormal frame introduced in
the problem setting, with
\(\operatorname{range}(Q_0)=\operatorname{Null}(Z^T)\).

Choose a local analytic orthonormal frame
\(V(\sigma)\in\R^{m\times q_0}\) whose columns span the noisy invariant
subspace descending from this clean null space and whose initial value is
\(V(0)=Q_0\).  The matrix
\begin{equation}
 \Lambda(\sigma)
   :=V(\sigma)^TA(\sigma)V(\sigma)\in\R^{q_0\times q_0}
\end{equation}
is the symmetric matrix representing the restriction of \(A(\sigma)\) to
that moving frame; it need not be diagonal in the chosen gauge.  Expand both
matrix-valued functions as
\begin{align}
    V(\sigma)&=\sum_{k=0}^{\infty}\sigma^kV_k,
    &V_0&=Q_0,
    &V_{\ell}&=0_{m\times q_0}\quad(\ell<0),\\
    \Lambda(\sigma)&=\sum_{k=0}^{\infty}\sigma^k\Lambda_k,
    &\Lambda_0&=0_{q_0\times q_0},
\end{align}
where every \(V_k\) is \(m\times q_0\) and every \(\Lambda_k\) is
\(q_0\times q_0\).  These definitions give
\begin{equation}
    A(\sigma)V(\sigma)=V(\sigma)\Lambda(\sigma),
    \qquad V(\sigma)^TV(\sigma)=I_{q_0}.
\end{equation}

\begin{prop}[All-order block recursion]
\label{prop:block-recursion}
For \(k\geq1\), let
\begin{align}
    \Lambda_k
      &=Q_0^TA_1V_{k-1}+Q_0^TA_2V_{k-2}
        -\sum_{j=1}^{k-1}(Q_0^TV_{k-j})\Lambda_j,
      \label{eq:block-lambda}\\
    V_k^{\perp}
      &=G\left(
          -A_1V_{k-1}-A_2V_{k-2}
          +\sum_{j=1}^{k-1}V_{k-j}\Lambda_j
        \right),
      \label{eq:block-perp}\\
    \Gamma_k
      &=-\frac12\sum_{j=1}^{k-1}V_j^TV_{k-j},
      \qquad
    V_k=V_k^{\perp}+Q_0\Gamma_k.
    \label{eq:block-parallel}
\end{align}
Then \(V^{[K]}(\sigma)=\sum_{k=0}^K\sigma^kV_k\) satisfies the invariant
subspace equation and orthonormality through order \(K\).  The sum defining
\(\Gamma_k\) is symmetric: transposition followed by the reindexing
\(j\leftrightarrow k-j\) leaves it unchanged.  Thus
\(Q_0^TV_k=\Gamma_k\) is symmetric, or equivalently
\(\operatorname{skew}(Q_0^TV_k)=0\); this choice fixes the otherwise
arbitrary internal rotation gauge.
\end{prop}

The proof is the matrix version of the proof of
\Cref{prop:scalar-recursion}.  In particular,
\begin{equation}
    \Lambda_1=0,
    \qquad
    V_1=-GA_1Q_0.
\end{equation}
The first nonzero splitting inside the clean null space occurs at second
order:
\begin{equation}
    \boxed{
    \Lambda_2=W
      =Q_0^T\bigl(A_2-A_1GA_1\bigr)Q_0.}
    \label{eq:effective-splitting}
\end{equation}

The same recursion has a basis-free projector form.  Suppressing the
full-nullity subscript \(q_0\), write
\(\widehat P(\sigma)=V(\sigma)V(\sigma)^T
  =\sum_{k\geq0}\sigma^kP_k\), with
\(P_0=Q_0Q_0^T\), \(P_{-1}=0\), and
\(P_0^\perp=I_m-P_0\).  For \(k\geq1\), set
\begin{equation}
    S_k=\sum_{j=1}^{k-1}P_jP_{k-j},
    \qquad
    C_k=[A_1,P_{k-1}]+[A_2,P_{k-2}],
\end{equation}
where \([B,C]=BC-CB\).  The identities
\(\widehat P(\sigma)^2=\widehat P(\sigma)\) and
\([A(\sigma),\widehat P(\sigma)]=0\) give
\begin{equation}
\boxed{
    P_k=-P_0S_kP_0+P_0^\perp S_kP_0^\perp
        -GC_kP_0+P_0C_kG.}
    \label{eq:projector-recursion}
\end{equation}
In particular, \(P_1=-GA_1P_0-P_0A_1G\).  This coordinate-free series is
the natural object to compare with the exact SVD projector.

\section{Finite-sample empirical risk}

For a fixed finite training realization \(\mathcal{T}=(X,\cE)\), define the
empirical residual risk of the retained \(q\)-dimensional estimator by
\begin{equation}
\boxed{
    R_{\mathrm{emp},q}(\sigma;\mathcal{T})
      =\frac{1}{\tau}
        \|\widetilde Z(\sigma)^T\widehat Q_q(\sigma)\|_F^2
      =\frac{1}{\tau}\operatorname{tr}\!\left(
         A(\sigma)\widehat P_q(\sigma)\right).}
    \label{eq:empirical-risk}
\end{equation}
The expected empirical error used in the portfolio template is the outer
training expectation
\begin{equation}
    \varepsilon_{\mathrm{emp},q}(\sigma)
       =\mathbb E_{\mathrm{tr}}
          [R_{\mathrm{emp},q}(\sigma;\mathcal{T})],
    \qquad \mathbb E_{\mathrm{tr}}=\mathbb E_{X,\cE}.
    \label{eq:expected-empirical-risk}
\end{equation}
For simple nullity,
\(R_{\mathrm{emp},1}(\sigma;\mathcal{T})=\lambda(\sigma)/\tau\), so the
compact reconstruction gives the exact empirical risk once \(\lambda\) is
supplied.

For complex \(s\) near zero, let \(P_q(s)\) denote the analytic continuation
of the bottom-\(q\) projector, chosen so that
\begin{equation}
    P_q(\sigma)=\widehat P_q(\sigma),
    \qquad
    P_q(s)=\sum_{k\geq0}s^kP_{q,k}
    \quad\text{for real \(\sigma\) sufficiently close to zero},
    \label{eq:analytic-projector-series}
\end{equation}
and set \(P_{q,k}=0\) for \(k<0\).  Since
\(A(s)=A_0+sA_1+s^2A_2\), the empirical-risk series is
\begin{equation}
    R_{\mathrm{emp},q}(s;\mathcal{T})
      =\sum_{k\geq0}e_{q,k}s^k,
    \qquad
\boxed{
    e_{q,k}=\frac{1}{\tau}\operatorname{tr}\!\left(
      A_0P_{q,k}+A_1P_{q,k-1}+A_2P_{q,k-2}\right).}
    \label{eq:empirical-risk-series}
\end{equation}
This is a fixed-realization expansion.  Passing the series through
\(\mathbb E_{\mathrm{tr}}\) requires a separate domination or uniform-
convergence argument.  When such an interchange is justified,
\begin{equation}
    \varepsilon_{\mathrm{emp},q}(\sigma)
      =\sum_{k\geq0}\bar e_{q,k}\sigma^k,
    \qquad
    \bar e_{q,k}=\mathbb E_{\mathrm{tr}}[e_{q,k}].
    \label{eq:expected-empirical-series}
\end{equation}

\begin{prop}[Second-order origin of the empirical ranking]
\label{prop:empirical-ranking}
Let
\(\lambda_{(0)}(\sigma)\leq\cdots\leq
  \lambda_{(m-1)}(\sigma)\)
be the increasingly ordered eigenvalues of \(A(\sigma)\), and let
\(P_{(i)}(\sigma)\) be rank-one eigenprojectors from an ordered
orthonormal eigenbasis.  When an eigenvalue is simple its projector is unique;
the associated unit eigenvector is unique up to sign.  For \(0\leq i<q_0\),
define
\begin{equation}
 R_{\mathrm{emp},i}^{\mathrm{ord}}(\sigma;\mathcal T)
   =\frac1\tau\operatorname{tr}\!\left(
      A(\sigma)P_{(i)}(\sigma)\right),
 \qquad
 \varepsilon_{\mathrm{emp},i}^{\mathrm{ord}}(\sigma)
   =\mathbb E_{\mathrm{tr}}
      [R_{\mathrm{emp},i}^{\mathrm{ord}}(\sigma;\mathcal T)].
 \label{eq:ordered-empirical-risk}
\end{equation}
Then, exactly for every realization,
\begin{equation}
 \boxed{R_{\mathrm{emp},i}^{\mathrm{ord}}(\sigma;\mathcal T)
       =\frac{\lambda_{(i)}(\sigma)}{\tau}},
 \qquad
 R_{\mathrm{emp},0}^{\mathrm{ord}}\leq\cdots\leq
 R_{\mathrm{emp},q_0-1}^{\mathrm{ord}}.
 \label{eq:exact-empirical-order}
\end{equation}
For every \(1\leq q\leq q_0\), these branch risks decompose the aggregate
risk, both before and after expectation:
\begin{equation}
 R_{\mathrm{emp},q}(\sigma;\mathcal T)
   =\sum_{i=0}^{q-1}R_{\mathrm{emp},i}^{\mathrm{ord}}(\sigma;\mathcal T),
 \qquad
 \varepsilon_{\mathrm{emp},q}(\sigma)
   =\sum_{i=0}^{q-1}\varepsilon_{\mathrm{emp},i}^{\mathrm{ord}}(\sigma).
 \label{eq:ordered-empirical-risk-sum}
\end{equation}
If \(\mu_0\leq\cdots\leq\mu_{q_0-1}\) are the ordered eigenvalues of
the effective matrix \(W\) in \eqref{eq:effective-splitting}, then
\begin{equation}
 \boxed{
 R_{\mathrm{emp},i}^{\mathrm{ord}}(\sigma;\mathcal T)
   =\frac{\sigma^2}{\tau}\mu_i+O_{\mathcal T}(\sigma^3),
   \qquad \sigma\downarrow0.}
 \label{eq:second-order-empirical-ranking}
\end{equation}
Thus \(W\), rather than an arbitrary clean null-space basis, governs the
first nonzero small-noise separations.  If \(W\) is simple, its locally
analytic labelled sheets coincide with \(\lambda_{(i)}\) for all sufficiently
small \(\sigma>0\).  At larger noise, an eigenvalue crossing may exchange
sheet labels; the parenthesized index always denotes instantaneous SVD order.

Under the Gaussian assumptions of \Cref{sec:equal-ranked-equal}, suppose
also that \(\tau\geq m\).  Then
\begin{equation}
 W\sim\operatorname{Wishart}_{q_0}(\tau-n,I_{q_0}),
 \qquad
 \varepsilon_{\mathrm{emp},i}^{\mathrm{ord}}(\sigma)
   =\frac{\sigma^2}{\tau}\,\overline\mu_i+o(\sigma^2),
 \quad \overline\mu_i=\mathbb E[\mu_i],
 \label{eq:expected-second-order-empirical-ranking}
\end{equation}
and
\(\overline\mu_0<\cdots<\overline\mu_{q_0-1}\).
Let \(\overline{\V{\mu}}=(\overline\mu_0,\ldots,
\overline\mu_{q_0-1})^T\).  Consequently, if
\(\V{e}_{\mathrm{emp}}\) collects the expected ordered empirical risks, then
\begin{equation}
 \frac{\varepsilon_{\mathrm{emp},i}^{\mathrm{ord}}(\sigma)}
      {\|\V{e}_{\mathrm{emp}}(\sigma)\|_2}
 \longrightarrow
 \frac{\overline\mu_i}{\|\overline{\V{\mu}}\|_2}
 \quad(\sigma\downarrow0),
 \label{eq:normalized-empirical-small-limit}
\end{equation}
which, when \(q_0>1\), is strictly ranked rather than uniform.  Moreover, for every
\(\sigma>0\) the Gaussian sample covariance has simple spectrum almost
surely, so the expected inequalities in \eqref{eq:exact-empirical-order}
are strict.
\end{prop}

\begin{proof}
The spectral identity
\(\operatorname{tr}(A P_{(i)})=\lambda_{(i)}\) proves
\eqref{eq:exact-empirical-order} and, after summing the bottom eigenvalues,
\eqref{eq:ordered-empirical-risk-sum}.  To expose the first nonzero term, let
\(Q_1\) span \(\operatorname{range}(Z)\), put
\(C=Q_1^TA_0Q_1>0\), and decompose a bottom eigenvector as
\(\V{u}=Q_0\V{\phi}+Q_1\V{y}\).  The range equation gives
\begin{equation}
 \V{y}=-\sigma C^{-1}Q_1^TA_1Q_0\V{\phi}+O(\sigma^2).
 \label{eq:empirical-range-elimination}
\end{equation}
Substitution into the null-space equation, using
\(G=Q_1C^{-1}Q_1^T\), gives
\begin{equation}
 \sigma^2Q_0^T(A_2-A_1GA_1)Q_0\V{\phi}
    =\lambda(\sigma)\V{\phi}+O(\sigma^3).
 \label{eq:empirical-second-order-eigenproblem}
\end{equation}
Hermitian degenerate perturbation therefore gives
\(\lambda_{(i)}(\sigma)/\sigma^2\to\mu_i\), with the remainder in
\eqref{eq:second-order-empirical-ranking}; if \(W\) is simple, this also
identifies the local analytic sheets.

Under the stated rank assumptions,
\(\operatorname{Null}(Z^T)=\operatorname{Null}(H^T)\), so \(Q_0\) may be
chosen deterministically given \(H\).  For Gaussian training,
\eqref{eq:W-residual-form} gives, with \(Y=Q_0^T\cE\),
\begin{equation}
 W=Y(I_\tau-\Pi_X)Y^T.
\end{equation}
The entries of \(Y\) are independent standard Gaussians, while, conditionally
on \(X\), \(I_\tau-\Pi_X\) is an orthogonal projector of rank \(\tau-n\).
Rotational invariance therefore gives the
Wishart law in \eqref{eq:expected-second-order-empirical-ranking}.  Its
spectrum is simple and positive almost surely when
\(\tau-n\geq q_0\), hence its ordered eigenvalue means are strictly
increasing: for \(i<j\), the integrable gap \(\mu_j-\mu_i\) is positive
almost surely and therefore has positive expectation.  Finally, the
min--max principle on the trial subspace
\(\operatorname{range}(Q_0)\) gives
\begin{equation}
 0\leq\lambda_{(i)}(\sigma)
 \leq\lambda_{(q_0-1)}(\sigma)
 \leq\sigma^2\|Q_0^T\cE\|_2^2
 \leq\sigma^2\|\cE\|_2^2,
 \qquad 0\leq i<q_0.
\end{equation}
The upper bound is integrable, so dominated convergence justifies the
expected limit and then \eqref{eq:normalized-empirical-small-limit}.  For
\(\sigma>0\), the columns of \(\widetilde Z\) are independent
\(\mathcal N(0,HH^T+\sigma^2I_m)\) vectors.  When \(\tau\geq m\), their
central Wishart matrix is positive definite with simple spectrum almost
surely.  Thus the pointwise inequalities are strict almost surely, and so
are their expectations.
\end{proof}

\section{Finite-sample generalization risk}

\subsection{Conditional compact and series forms}

Let a new measurement, independent of the training matrix, be
\begin{equation}
    \V{z}_{\mathrm{te}}
      =H\V{x}_{\mathrm{te}}+\sigma\V{e}_{\mathrm{te}},
    \qquad
    \mathbb E[\V{x}_{\mathrm{te}}\V{x}_{\mathrm{te}}^T]=I_n,
    \qquad
    \mathbb E[\V{e}_{\mathrm{te}}\V{e}_{\mathrm{te}}^T]=I_m,
    \label{eq:test-model}
\end{equation}
where the two test variables are independent of each other and of the
training realization.  Assume here that
\(\operatorname{rank}(X)=\operatorname{rank}(H)=n\), so the clean null
spaces of \(Z^T\) and \(H^T\) coincide.  Conditional on
\(\mathcal{T}=(X,\cE)\), the simple-nullity population risk is
\begin{equation}
    R_{\mathrm{gen}}(\sigma;\mathcal{T})
      =\mathbb E\!\left[
        (\widehat{\V{\eta}}(\sigma)^T\V{z}_{\mathrm{te}})^2
        \mathrel{|}\mathcal{T}
      \right].
\end{equation}
Writing \(B=HH^T\), using \(H^T\V{\eta}=0\), and recalling that
\(\widehat{\V{\eta}}\) is a unit vector gives the exact identity
\begin{equation}
\boxed{
    R_{\mathrm{gen}}(\sigma;\mathcal{T})
      =\widehat{\V{\eta}}(\sigma)^T
         (B+\sigma^2I_m)\widehat{\V{\eta}}(\sigma)
      =\sigma^2+\V{\epsilon}(\sigma)^TB\V{\epsilon}(\sigma).}
    \label{eq:generalization-compact-identity}
\end{equation}
Consequently, substitution of the compact vector in
\eqref{eq:compact-error} yields the compact scalar form
\begin{equation}
\boxed{
    R_{\mathrm{gen}}^{\mathrm{comp}}(\sigma;\mathcal{T})
      =\sigma^2+(\alpha\V{w}-\V{v})^T
         B(\alpha\V{w}-\V{v}).}
    \label{eq:generalization-compact}
\end{equation}
It is exact whenever the compact eigenvector formula is exact.

For the strict series, substitute
\(\V{\epsilon}(\sigma)=\sum_{k\geq1}\sigma^k\V{\eta}_k\) into
\eqref{eq:generalization-compact-identity}.  Then
\begin{equation}
    R_{\mathrm{gen}}(\sigma;\mathcal{T})
      =\sum_{\ell\geq0}g_\ell\sigma^\ell,
    \qquad g_0=g_1=0,
\end{equation}
with the closed coefficient convolution
\begin{equation}
\boxed{
    g_\ell
      =\boldsymbol{1}_{\{\ell=2\}}
       +\sum_{j=1}^{\ell-1}
          \V{\eta}_j^TB\V{\eta}_{\ell-j},
      \qquad \ell\geq2.}
    \label{eq:generalization-series}
\end{equation}
Thus an order-\(K\) risk approximation retains only
\(\sum_{\ell=0}^Kg_\ell\sigma^\ell\); products of vector coefficients whose
total degree exceeds \(K\) must not be included.  This consistent truncation
is important when the recursion is compared with the exact SVD risk.

For a retained \(q\)-dimensional subspace, the corresponding aggregate risk
has the basis-free form
\begin{equation}
\boxed{
    R_{\mathrm{gen},q}(\sigma;\mathcal{T})
       =q\sigma^2+\operatorname{tr}\!\left(
          B\widehat P_q(\sigma)\right).}
    \label{eq:generalization-projector}
\end{equation}
Using the analytic continuation in
\eqref{eq:analytic-projector-series}, its coefficients are
\begin{equation}
    g_{q,k}=q\boldsymbol{1}_{\{k=2\}}
       +\operatorname{tr}(BP_{q,k}).
    \label{eq:generalization-projector-series}
\end{equation}
Thus
\begin{equation}
    R_{\mathrm{gen},q}^{[K]}(\sigma;\mathcal{T})
      =\sum_{k=0}^{K}g_{q,k}\sigma^k,
    \qquad
    \delta_{\mathrm{gen},q}^{[K]}(\sigma)
      =\left|R_{\mathrm{gen},q}^{[K]}(\sigma;\mathcal{T})
             -R_{\mathrm{gen},q}(\sigma;\mathcal{T})\right|.
    \label{eq:generalization-series-error}
\end{equation}
The template's expected generalization error is
\begin{equation}
    \varepsilon_{\mathrm{gen},q}(\sigma)
      =\mathbb E_{\mathrm{tr}}
         [R_{\mathrm{gen},q}(\sigma;\mathcal{T})]
      =q\sigma^2+\operatorname{tr}\!\left(
         B\,\mathbb E_{\mathrm{tr}}[\widehat P_q(\sigma)]\right).
    \label{eq:expected-generalization-risk}
\end{equation}
For \(q=1\), this becomes the exact compact expectation
\begin{equation}
    \varepsilon_{\mathrm{gen}}(\sigma)
      =\sigma^2+
        \mathbb E_{\mathrm{tr}}[\V{\epsilon}(\sigma)^T
          B\V{\epsilon}(\sigma)].
    \label{eq:expected-generalization-compact}
\end{equation}
As for the empirical risk, our compact and series identities hold
conditionally for every fixed training realization.  If a training ensemble
has a common convergence disk and integrable coefficient bounds that justify
termwise expectation, then the expected-error series is
\begin{equation}
    \varepsilon_{\mathrm{gen},q}(\sigma)
       =\sum_{k\geq0}\bar g_{q,k}\sigma^k,
    \qquad
    \bar g_{q,k}=\mathbb E_{\mathrm{tr}}[g_{q,k}].
    \label{eq:expected-generalization-series}
\end{equation}
Without those conditions---in particular for an unbounded Gaussian noise
ensemble---the fixed-realization series cannot automatically be averaged
term by term.  The recursion notebooks therefore verify the conditional risk
\(R_{\mathrm{gen},q}\).  The ordered-branch experiment introduced below is
different: it estimates the outer expectation directly by Monte Carlo and
uses a local polynomial fit to estimate its low-noise coefficients.

\subsection{Expected ordered-branch generalization: the equal--ranked--equal law}
\label{sec:equal-ranked-equal}

The aggregate projector in \eqref{eq:generalization-projector} deliberately
forgets rotations within the retained subspace.  The SVD nevertheless
orders the individual vectors by their singular values.  We now ask when
that empirical ordering is also visible in expected generalization risk.

Assume in this subsection that \(H\) has full column rank, that \(X\) and
\(\cE\) have independent standard-real-Gaussian entries, and that
\(\tau>n+1\).  Suppose also that the eigenvalues of the effective splitting
matrix \(W\) in \eqref{eq:effective-splitting} are simple.  Write
\begin{equation}
    W=\Phi\operatorname{diag}(\mu_0,\ldots,\mu_{q_0-1})\Phi^T,
    \qquad \mu_0<\cdots<\mu_{q_0-1},
    \label{eq:ordered-W}
\end{equation}
and let \(P_i^{W}(s)\) be the analytic rank-one projector on the spectral
sheet whose small-noise limit is
\(Q_0\V{\phi}_i\V{\phi}_i^TQ_0^T\).  Thus the branches are selected by
\(W\), rather than by an arbitrary basis of the clean null space.  For real
\(\sigma>0\) sufficiently close to zero,
\(P_i^W(\sigma)=P_{(i)}(\sigma)\), where \(P_{(i)}\) denotes the
instantaneous increasingly ordered SVD projector used in the experiments.
The notebook label ``Rank \(i+1\)'' is this instantaneous SVD order; it is
not an intrinsic dimension of a null vector.

Define the conditional and expected ordered-direction risks by
\begin{align}
    R_{\mathrm{gen},i}^{\mathrm{ord}}(\sigma;\mathcal T)
       &=\sigma^2+\operatorname{tr}\!\left(BP_{(i)}(\sigma)\right),
       \label{eq:ordered-conditional-risk}\\
    \varepsilon_{\mathrm{gen},i}^{\mathrm{ord}}(\sigma)
       &=\mathbb E_{\mathrm{tr}}
          [R_{\mathrm{gen},i}^{\mathrm{ord}}(\sigma;\mathcal T)].
       \label{eq:ordered-expected-risk}
\end{align}
For every retained rank $1\leq q\leq q_0$, the ordered branch risks add to
the aggregate risk:
\begin{equation}
    \sum_{i=0}^{q-1}
       \varepsilon_{\mathrm{gen},i}^{\mathrm{ord}}(\sigma)
       =\varepsilon_{\mathrm{gen},q}(\sigma).
    \label{eq:ordered-risk-sum}
\end{equation}

Let \(\V{g}(\sigma)\) collect these \(q_0\) expected risks and set
\begin{equation}
    p_i(\sigma)=
       \frac{\varepsilon_{\mathrm{gen},i}^{\mathrm{ord}}(\sigma)}
            {\|\V{g}(\sigma)\|_2},
    \qquad
    \Delta_{\mathrm{rank}}(\sigma)
       =\max_i p_i(\sigma)-\min_i p_i(\sigma).
    \label{eq:rank-spread}
\end{equation}
At \(\sigma=0\) both numerator and denominator vanish; we define
\(p_i(0)=1/\sqrt{q_0}\) and
\(\Delta_{\mathrm{rank}}(0)=0\) by their small-noise limits.  The
normalization in \eqref{eq:rank-spread} is applied \emph{after} the
training expectation, exactly as in the experiment.

For a visibility tolerance \(\delta_{\mathrm{vis}}>0\), put
\begin{equation}
    \mathcal C_{\mathrm{vis}}
       =\{\sigma>0:
          \Delta_{\mathrm{rank}}(\sigma)>\delta_{\mathrm{vis}}\}.
    \label{eq:rank-crossover-set}
\end{equation}
If this set is one interval, denote its endpoints by
\(\sigma_{\mathrm{lo}}(\delta_{\mathrm{vis}})\) and
\(\sigma_{\mathrm{hi}}(\delta_{\mathrm{vis}})\).  Multiple components must be reported
separately; the definition does not force a three-regime picture.

The expansions below are statements about ensemble expectations.  We assume
that their displayed remainders hold in \(L^1\), so that the finite
expansions may be averaged.  Pointwise analyticity for each realization is
not enough to supply this property for an unbounded Gaussian ensemble.

\begin{prop}[Small-noise equalization and corrected lower crossover]
\label{prop:small-noise-equalization}
Let \(h_\star=\|H\|_2\) and \(\zeta=\sigma/h_\star\).  Assume that the
ordered expected risks may be expanded through sixth order and that their
maximizing and minimizing branches are unique and do not switch near zero.
Then
\begin{equation}
    \varepsilon_{\mathrm{gen},i}^{\mathrm{ord}}(\sigma)
      =h_\star^2\zeta^2
        \left(a+b_i\zeta^2+c_i\zeta^4+O(\zeta^6)\right),
    \qquad
    \boxed{a=1+\frac{n}{\tau-n-1}},
    \label{eq:ordered-small-risk}
\end{equation}
and every branch has the same leading coefficient \(a\).  Equivalently,
\begin{equation}
 \boxed{
 \varepsilon_{\mathrm{gen},i}^{\mathrm{ord}}(\sigma)
   =\left(1+\frac{n}{\tau-n-1}\right)\sigma^2+O(\sigma^4).}
 \label{eq:ordered-small-leading-law}
\end{equation}
The first rank-dependent contribution is contained in the
\(O(\sigma^4)\) term.  If
\(i_+=\arg\max_i b_i\), \(i_-=\arg\min_i b_i\), and
\(\bar b=q_0^{-1}\sum_i b_i\), then
\begin{equation}
    \Delta_{\mathrm{rank}}(\sigma)
       =C_2\zeta^2+C_4\zeta^4+O(\zeta^6),
    \label{eq:rank-small-spread}
\end{equation}
where
\begin{align}
    C_2&=\frac{b_{i_+}-b_{i_-}}{a\sqrt{q_0}},
       \label{eq:C2-definition}\\
    C_4&=\frac{1}{\sqrt{q_0}}
       \left[
       \frac{c_{i_+}-c_{i_-}}{a}
       -\frac{\bar b(b_{i_+}-b_{i_-})}{a^2}
       \right].
       \label{eq:C4-definition}
\end{align}
If \(C_2>0\), then, for sufficiently small \(\delta_{\mathrm{vis}}\) such
that the discriminant below is nonnegative and the positive root remains in
this expansion and no-switch neighborhood, the fourth-order corrected lower
crossover is
\begin{equation}
\boxed{
    \sigma_{\mathrm{lo}}(\delta_{\mathrm{vis}})
      \simeq h_\star
       \left(
       \frac{2\delta_{\mathrm{vis}}}
       {C_2+\sqrt{C_2^2+4C_4\delta_{\mathrm{vis}}}}
       \right)^{1/2}.}
    \label{eq:corrected-lower-crossover}
\end{equation}
For \(C_4=0\), this reduces to
\(h_\star\sqrt{\delta_{\mathrm{vis}}/C_2}\).
\end{prop}

\begin{proof}
Let \(\V{q}_i=Q_0\V{\phi}_i\).  The first-order component of the \(i\)-th
branch outside the clean null space is
\begin{equation}
    \V{v}_{i,1}=-(Z^T)^+\cE^T\V{q}_i.
    \label{eq:ordered-first-vector}
\end{equation}
Put
\(\Pi_X=X^T(XX^T)^{-1}X\).  Since \(H\) has full column rank and \(X\)
has full row rank almost surely,
\begin{equation}
    Z^T(ZZ^T)^+Z=\Pi_X,
\end{equation}
and the effective splitting has the equivalent form
\begin{equation}
    W=Q_0^T\cE(I_\tau-\Pi_X)\cE^TQ_0.
    \label{eq:W-residual-form}
\end{equation}
Thus \(\V{q}_i\) is determined by the residual Gaussian block
\(\cE(I_\tau-\Pi_X)\), whereas the first-order signal leakage in
\eqref{eq:ordered-first-vector} uses the independent block
\(\cE\Pi_X\).  Moreover,
\begin{equation}
    H^T\V{v}_{i,1}
      =-(XX^T)^{-1}X\cE^T\V{q}_i.
\end{equation}
Conditioning on \(X\) and on the residual block therefore gives
\begin{equation}
    \mathbb E\!\left[
      \V{v}_{i,1}^TB\V{v}_{i,1}
      \mathrel{|}X,W\right]
       =\operatorname{tr}\!\left((XX^T)^{-1}\right),
\end{equation}
independently of \(i\).  Since
\(XX^T\sim\operatorname{Wishart}_n(\tau,I_n)\),
\begin{equation}
    \mathbb E[(XX^T)^{-1}]
       =\frac{I_n}{\tau-n-1},
\end{equation}
which proves the value of \(a\) after adding the unit test-noise term.

The transformation \(\cE\mapsto-\cE\) preserves the Gaussian ensemble and
the \(W\)-labels but changes the sign of the perturbation parameter.  Hence
the expected odd coefficients vanish.  If
\(P_i^W(s)=\sum_{k\ge0}s^kP_{i,k}^W\) is the \(W\)-labelled branch
projector, then the main generalization series gives
\begin{equation}
    b_i=h_\star^2\operatorname{tr}\!\left(
       B\,\mathbb E_{\mathrm{tr}}[P_{i,4}^W]\right),
    \qquad
    c_i=h_\star^4\operatorname{tr}\!\left(
       B\,\mathbb E_{\mathrm{tr}}[P_{i,6}^W]\right).
    \label{eq:bc-projector-coefficients}
\end{equation}
These are the fourth- and sixth-order coefficients of the existing
projector-risk expansion, after the scale \(h_\star\) is extracted.

Finally, expanding the common denominator in
\eqref{eq:rank-spread} gives
\begin{equation}
 p_i(\sigma)=\frac{1}{\sqrt{q_0}}
 \left[
 1+\frac{b_i-\bar b}{a}\zeta^2
 +\left\{
   \frac{c_i-\bar c}{a}
   -\frac{\bar b(b_i-\bar b)}{a^2}
   +\kappa
  \right\}\zeta^4
 +O(\zeta^6)
 \right],
 \label{eq:normalized-risk-expansion}
\end{equation}
where \(\bar c=q_0^{-1}\sum_i c_i\) and \(\kappa\) is independent of \(i\).
Taking the range yields \eqref{eq:C2-definition}--
\eqref{eq:C4-definition}.  Solving
\(\delta_{\mathrm{vis}}=C_2\zeta^2+C_4\zeta^4\) for the root continuous
from zero gives
\eqref{eq:corrected-lower-crossover}.
\end{proof}

The coefficient fit in the experiment uses
\begin{equation}
 \frac{\varepsilon_{\mathrm{gen},i}^{\mathrm{ord}}
       (h_\star\zeta)}{h_\star^2\zeta^2}
   =\alpha_i+b_i\zeta^2+c_i\zeta^4+e_i\zeta^6.
 \label{eq:rank-coefficient-fit}
\end{equation}
The separate intercepts \(\alpha_i\) absorb Monte Carlo error in the
theoretically common \(a\); the \(e_i\)-term is a nuisance coefficient that
prevents the finite fitting interval from contaminating \(b_i\) and \(c_i\).
Only \(b_i\) and \(c_i\) enter the corrected threshold.

\begin{prop}[Large-noise equalization and upper crossover]
\label{prop:large-noise-equalization}
Assume additionally that \(\tau\geq m\), and that the ordered projector
expansion at inverse noise has an \(L^1\) remainder through fourth order.
Set \(\overline Z=Z/h_\star\), \(\overline B=B/h_\star^2\), and
\(t=\zeta^{-1}\).  Let
\begin{equation}
 \Pi_i(t)=\sum_{k\geq0}t^k\Pi_{i,k}^{(\infty)}
 \label{eq:large-noise-projector-series}
\end{equation}
be the ordered rank-one eigenprojector of
\((\cE+t\overline Z)(\cE+t\overline Z)^T\).  Then
\begin{equation}
 \varepsilon_{\mathrm{gen},i}^{\mathrm{ord}}(\sigma)
   =h_\star^2\left(
      \zeta^2+c_\infty+d_i^{(\infty)}\zeta^{-2}
      +O(\zeta^{-4})\right),
 \label{eq:ordered-large-risk}
\end{equation}
where
\begin{equation}
 c_\infty=\frac{\operatorname{tr}(\overline B)}{m}
   =\frac{\|H\|_F^2}{m h_\star^2},
 \qquad
 d_i^{(\infty)}=\operatorname{tr}\!\left(
   \overline B\,\mathbb E_{\mathrm{tr}}
      [\Pi_{i,2}^{(\infty)}]\right).
 \label{eq:large-risk-coefficients}
\end{equation}
In unscaled variables, the common leading law is
\begin{equation}
 \boxed{
 \varepsilon_{\mathrm{gen},i}^{\mathrm{ord}}(\sigma)
   =\sigma^2+\frac{\|H\|_F^2}{m}+O(\sigma^{-2}).}
 \label{eq:ordered-large-leading-law}
\end{equation}
Consequently,
\begin{equation}
 \Delta_{\mathrm{rank}}(\sigma)
   =C_H\zeta^{-4}+O(\zeta^{-6}),
 \qquad
 C_H=\frac{\max_i d_i^{(\infty)}-\min_i d_i^{(\infty)}}
          {\sqrt{q_0}}.
 \label{eq:rank-large-spread}
\end{equation}
When \(C_H>0\), the leading upper visibility crossover is
\begin{equation}
 \boxed{
 \sigma_{\mathrm{hi}}(\delta_{\mathrm{vis}})
   \simeq h_\star
       \left(\frac{C_H}{\delta_{\mathrm{vis}}}\right)^{1/4}.}
 \label{eq:upper-crossover}
\end{equation}
\end{prop}

\begin{proof}
The noise factor \(\sigma^2\) changes the eigenvalues of the noisy Gram
matrix but not its eigenvectors, so the large-noise problem is exactly the
small-\(t\) projector problem in
\eqref{eq:large-noise-projector-series}.  At \(t=0\),
\(\cE\cE^T\) is a real isotropic Wishart matrix.  Its ordered eigenvectors
are marginally Haar distributed and independent of its ordered eigenvalues;
hence
\begin{equation}
 \mathbb E_{\mathrm{tr}}[\Pi_{i,0}^{(\infty)}]
    =\frac{I_m}{m}
 \label{eq:haar-projector-mean}
\end{equation}
for every rank \(i\).  This gives the common term \(c_\infty\).
The symmetry \(X\mapsto-X\) preserves the training ensemble and maps
\(t\mapsto-t\), so all expected odd inverse-noise coefficients vanish.
The first possibly rank-dependent term is therefore the displayed
\(t^2\)-coefficient \(d_i^{(\infty)}\), proving
\eqref{eq:ordered-large-risk}.

The common leading norm is
\(\|\V{g}(\sigma)\|_2
  =h_\star^2\sqrt{q_0}\,\zeta^2+O(1)\).  Taking the range of the normalized
risks therefore divides the raw difference
\(h_\star^2(\max_i d_i^{(\infty)}-\min_i d_i^{(\infty)})
\zeta^{-2}+O(\zeta^{-4})\)
by this common norm.  This proves \eqref{eq:rank-large-spread}; solving its
leading term for \(\delta_{\mathrm{vis}}\) gives
\eqref{eq:upper-crossover}.
\end{proof}

\begin{coro}[Equal--ranked--equal profile]
\label{cor:equal-ranked-equal}
Under the assumptions of the two propositions,
\begin{equation}
 p_i(\sigma)\longrightarrow\frac{1}{\sqrt{q_0}}
 \quad\text{as }\sigma\downarrow0
 \quad\text{and as }\sigma\to\infty.
 \label{eq:endpoint-uniform-profile}
\end{equation}
Thus rank dependence can first appear at order \(\zeta^2\) in the normalized
small-noise profile and is \(O(\zeta^{-4})\) in the normalized
large-noise profile.  If \(\Delta_{\mathrm{rank}}\) is continuous and not
identically zero, it has a positive maximum at a finite, nonzero noise
level.
\end{coro}

\subsubsection{Why the intermediate generalization risks are ranked}

There is no pointwise implication from empirical ordering to generalization
ordering: two sample eigenvectors can exchange their population leakage in
an individual trial.  Under Gaussian training, however, the complete noisy
data matrix has the exact distribution
\begin{equation}
 \widetilde Z(\sigma)_{:t}\sim
 \mathcal N(0,\Sigma_\sigma),
 \qquad
 \Sigma_\sigma=B+\sigma^2I_m,
 \qquad
 A(\sigma)\sim
 \operatorname{Wishart}_m(\tau,\Sigma_\sigma).
 \label{eq:noisy-Wishart-reduction}
\end{equation}
For the next two results, extend the ordered-risk notation to
\(0\leq i<m\) using the instantaneous ordered sample eigenvectors.
Consequently, if \(\widehat{\V{u}}_{(i)}\) is the sample eigenvector of
increasing rank \(i\), then
\begin{equation}
 \varepsilon_{\mathrm{gen},i}^{\mathrm{ord}}(\sigma)
  =\sigma^2+
    \mathbb E[\widehat{\V{u}}_{(i)}^TB\widehat{\V{u}}_{(i)}]
  =\mathbb E[\widehat{\V{u}}_{(i)}^T
       \Sigma_\sigma\widehat{\V{u}}_{(i)}].
 \label{eq:ordered-population-Rayleigh}
\end{equation}
The next theorem proves the ranking exactly when the nonzero population
eigenvalues are equal.

\begin{theorem}[Exact rank ordering for an isotropic signal subspace]
\label{thm:isotropic-middle-ranking}
Suppose \(B=\beta\Pi\), where \(\beta>0\) and \(\Pi\) is an orthogonal
projector with \(0<\operatorname{rank}(\Pi)<m\), and suppose
\(\tau\geq m\).  For every finite \(\sigma>0\), all ordered expected
generalization risks are strictly increasing:
\begin{equation}
 \varepsilon_{\mathrm{gen},0}^{\mathrm{ord}}(\sigma)
  <\varepsilon_{\mathrm{gen},1}^{\mathrm{ord}}(\sigma)
  <\cdots<
  \varepsilon_{\mathrm{gen},m-1}^{\mathrm{ord}}(\sigma).
 \label{eq:isotropic-strict-risk-order}
\end{equation}
In particular, this holds for the bottom \(q_0=m-n\) branches.
\end{theorem}

\begin{proof}
Condition on the increasingly ordered, almost surely distinct eigenvalues
\(\lambda_{(0)}<\cdots<\lambda_{(m-1)}\) of \(A\), and write
\(A=U\operatorname{diag}(\lambda_{(i)})U^T\).  The Wishart density in
\eqref{eq:noisy-Wishart-reduction}, relative to Haar measure on \(U\), is
proportional to
\begin{equation}
 \exp\!\left\{
   \kappa_\sigma\sum_{k=0}^{m-1}\lambda_{(k)}w_k(U)
 \right\},
 \qquad
 w_k(U)=\V{u}_{(k)}^T\Pi\V{u}_{(k)},
 \qquad
 \kappa_\sigma=
   \frac{\beta}{2\sigma^2(\sigma^2+\beta)}>0.
 \label{eq:isotropic-angular-density}
\end{equation}
Indeed,
\((\sigma^2I+\beta\Pi)^{-1}
 =\sigma^{-2}I-
   \beta[\sigma^2(\sigma^2+\beta)]^{-1}\Pi\), and the first term is
constant after conditioning on the sample eigenvalues.

Fix \(i<j\), interchange columns \(i\) and \(j\) of \(U\), and denote the
result by \(U^{ij}\).  With
\(d=w_j(U)-w_i(U)\), the two log weights differ by
\begin{equation}
 F(U)-F(U^{ij})
   =\kappa_\sigma
      (\lambda_{(j)}-\lambda_{(i)})d.
 \end{equation}
Pairing every \(U\) with \(U^{ij}\) in the conditional integral for
\(\mathbb E[w_j-w_i\mid\lambda]\) gives an integrand proportional to
\begin{equation}
 d\{e^{F(U)}-e^{F(U^{ij})}\}\geq0.
 \end{equation}
It is positive on a set of positive Haar measure because the sample
eigenvalues are distinct and \(\Pi\) is nontrivial.  Hence
\(\mathbb E[w_j\mid\lambda]>
  \mathbb E[w_i\mid\lambda]\).  Averaging over the eigenvalues and using
\eqref{eq:ordered-population-Rayleigh} proves
\eqref{eq:isotropic-strict-risk-order}.
\end{proof}

The theorem also explains why the ranking is most visible in the middle.
It is strict at every finite positive noise level, but
\Cref{prop:small-noise-equalization,prop:large-noise-equalization} show that
its normalized gaps are \(O(\zeta^2)\) as \(\zeta\downarrow0\) and
\(O(\zeta^{-4})\) as \(\zeta\to\infty\).  Thus ``equal'' at the two endpoints
means asymptotic equality, not exact equality at a finite \(\sigma\).

For unequal signal eigenvalues the angular tilt is instead
\begin{equation}
 C_\sigma=\sigma^{-2}I_m-(B+\sigma^2I_m)^{-1},
 \label{eq:unequal-angular-tilt}
\end{equation}
and the paired swap has the sign of the product of the \(B\)- and
\(C_\sigma\)-Rayleigh-quotient differences.  Although these two matrices
commute and their eigenvalues have the same order, the two Rayleigh
differences need not have the same sign when there are three or more distinct
population eigenvalues.  The isotropic proof therefore does not establish a
universal unequal-spike theorem.  The following identity gives an exact,
checkable mechanism without assuming equal spikes.

\begin{prop}[Population-overlap criterion]
\label{prop:overlap-ranking}
Let
\(B=\sum_{k=1}^m\beta_k\V{v}_k\V{v}_k^T\), with
\(0\leq\beta_1\leq\cdots\leq\beta_m\), and define the expected squared
overlaps
\begin{equation}
 \Omega_{ik}(\sigma)
   =\mathbb E[(\V{v}_k^T\widehat{\V{u}}_{(i)})^2].
 \label{eq:expected-overlap-matrix}
\end{equation}
For any \(i<j\), the exact expected-risk gap is
\begin{equation}
 \boxed{
 \varepsilon_{\mathrm{gen},j}^{\mathrm{ord}}(\sigma)
  -\varepsilon_{\mathrm{gen},i}^{\mathrm{ord}}(\sigma)
  =\sum_{\ell=2}^m(\beta_\ell-\beta_{\ell-1})
     \sum_{k=\ell}^m
       (\Omega_{jk}-\Omega_{ik}).}
 \label{eq:overlap-gap-identity}
\end{equation}
Therefore the risk is ordered whenever every cumulative overlap with the
upper population eigenspaces is nondecreasing with sample rank; the order is
strict if at least one such inequality is strict across a nonzero population
eigenvalue gap.
\end{prop}

\begin{proof}
Equation \eqref{eq:ordered-population-Rayleigh} gives the gap as
\(\sum_k\beta_k(\Omega_{jk}-\Omega_{ik})\).  Both rows of
\(\Omega\) sum to one.  Summation by parts therefore yields
\eqref{eq:overlap-gap-identity}, and the stated sign condition follows.
\end{proof}

The infinite-order expansion supplies a complementary local certificate.
If, for adjacent ranks, the remainder in
\eqref{eq:ordered-small-risk} obeys
\(|R_i(\zeta)|\leq M_i\zeta^6\), then the exact adjacent ordering is
guaranteed wherever
\begin{equation}
 (b_{i+1}-b_i)+(c_{i+1}-c_i)\zeta^2
   >(M_i+M_{i+1})\zeta^4.
 \label{eq:series-ranking-certificate}
\end{equation}
Thus the fourth-order coefficient explains the onset and direction of the
small-to-intermediate ranking, while higher orders bound the interval over
which that conclusion is rigorous.  The overlap identity continues to apply
outside the Taylor disk.

\begin{observation}[Scope of the middle-regime result]
The isotropic-subspace result is an exact ensemble theorem.  For a general
unequal-spike matrix \(H\), \eqref{eq:overlap-gap-identity} is an exact
sufficient criterion, but endpoint equalization alone does not imply its
hypotheses, unimodality of \(\Delta_{\mathrm{rank}}\), or a single connected
visibility interval.  The unequal-spike experiment below verifies the
criterion's mechanism and supplies a simultaneous \(99\%\) Monte Carlo
confidence certificate for positive expected risk gaps at one representative
middle-noise point.

The visibility crossings depend on \(\delta_{\mathrm{vis}}\) and on an
ensemble-averaged normalized risk.  They are therefore neither the
fixed-realization Taylor radius \(\rho_q\) nor an exact phase-transition or
BBP threshold.  The formulas are asymptotic crossing approximations unless
the omitted remainders are separately bounded.
\end{observation}

\section{Convergence radius and exceptional points}

Analytic perturbation theory guarantees a convergent local expansion for an
isolated eigenvalue or isolated eigenvalue cluster \cite{kato1995}.  Let
\(\gamma\) be the smallest positive eigenvalue of \(A_0\).  A norm condition
such as
\begin{equation}
    \|\sigma A_1+\sigma^2A_2\|_2
      <\tfrac12\gamma
    \label{eq:sufficient-gap-check}
\end{equation}
is a convenient sufficient small-noise check, but it is not the exact
convergence radius.

To obtain the exact algebraic candidates, analytically continue the real
noise scale \(\sigma\) to \(s\in\C\), set
\(A(s)=A_0+sA_1+s^2A_2\), and define
\begin{equation}
    p(\lambda,s)=\det(\lambda I_m-A(s)),
    \qquad
    D(s)=\operatorname{disc}_{\lambda}p(\lambda,s).
    \label{eq:discriminant}
\end{equation}
A root \(s_\star\) of \(D\) is an exceptional-point candidate at which two
eigenvalue sheets meet.  However, not every collision limits a grouped
spectral projector.  A collision internal to the retained cluster, or
internal to its complement, does not destroy the separation between the two
groups.

When \(q_0>1\), the discriminant has an exact factor \(s^\nu\) caused by the
clean zero cluster.  This degeneracy is removable after the branches are
labelled by their first nonzero effective splitting.  Define
\begin{equation}
    D_{\mathrm{red}}(s)=D(s)/s^\nu,
    \qquad D_{\mathrm{red}}(0)\ne0,
\end{equation}
with \(\nu=0\) when no zero factor is present.

Label the eigenvalue sheets \(\lambda_a(s)\), and let \(\cI_q\) be the set
of the \(q\) retained labels.  If sheets \(\lambda_a\) and \(\lambda_b\)
collide at \(s_\star\), the projector convergence radius is
\begin{equation}
    \boxed{
    \rho_q=\min\left\{
       |s_\star|:
       D_{\mathrm{red}}(s_\star)=0,
       \ \bigl|\{a,b\}\cap\cI_q\bigr|=1
    \right\}.}
    \label{eq:rho-formula}
\end{equation}
Thus \(\rho_q\) is computed independently of the projector-series-error
curve.  The curve is a numerical validation of \eqref{eq:rho-formula}, not
the source of \(\rho_q\).

\begin{theorem}[Gauge-invariant convergence statement]
\label{thm:projector-radius}
Assume the boundary collisions are generic.  The spectral projector onto the
retained branch set \(\cI_q\) is analytic for \(|s|<\rho_q\), and its Taylor
series about zero converges to the exact SVD projector in that disk.  A
compatible analytic frame, including the recursions above, has truncation
error \(O(|s|^{K+1})\).  At \(|s|>\rho_q\), increasing \(K\) cannot make the
Taylor series converge to the SVD projector.
\end{theorem}

\subsection{Radius of the generalization-risk series}

The scalar risk cannot have a singularity closer to zero than its projector,
because \eqref{eq:generalization-projector} is a linear functional of that
projector plus the entire polynomial \(q s^2\).  It may, however, have a
larger radius if that linear functional exactly cancels a projector
singularity.  At a simple exceptional point joining a target sheet \(a\) to
a complementary sheet \(b\), analytic continuation exchanges the two
projectors.  Define the scalar sheet difference
\begin{equation}
    \Delta_B(s)=\operatorname{tr}\!\left(
       B[P_a(s)-P_b(s)]\right).
\end{equation}
The exceptional point is invisible to the risk only in the special case
\(\Delta_B(s)\equiv0\) locally.  Equivalently, every half-integer term in
the local Puiseux expansion must cancel; cancellation of only its leading
coefficient is not enough.  Therefore
\begin{equation}
\boxed{
 \rho_{\mathrm{gen},q}=\min\!\left\{
 |s_\star|:\ D_{\mathrm{red}}(s_\star)=0,\
 |\{a,b\}\cap\cI_q|=1,\
 \Delta_B\not\equiv0
 \right\}.}
 \label{eq:generalization-radius}
\end{equation}
In particular, \(\rho_{\mathrm{gen},q}\geq\rho_q\), with equality
generically.

The generalization notebook uses the same realization as the projector
experiment in \texttt{seriesCheck2.ipynb}: random seed \(8\), \(m=6\),
\(n=1\), \(q_0=5\), and \(\tau=100\).  In particular,
\begin{equation}
H=(-0.49831520,-0.38318029,-0.39019345,-0.10079937,
   -0.66295623,-0.05415185)^T,
\end{equation}
and the two notebooks also share \(X\) and the standardized noise direction
\(\cE\).
Here \(B=HH^T\) and
\begin{equation}
    R_{\mathrm{gen},5}(s;\mathcal{T})
       =5s^2+\operatorname{tr}(BP_5(s)).
\end{equation}
The closest full-nullity target--complement exceptional point is
\begin{equation}
    s_{\star,5}=-1.16405872-0.43982812\,\mathrm{i},
    \qquad
    \boxed{\rho_{\mathrm{gen},5}=\rho_5
      =|s_{\star,5}|=1.24437996.}
    \label{eq:generalization-numerical-radius}
\end{equation}
The reduced generalization-risk plot uses the same second-order branch labels
and the same five radii reported later in \Cref{tab:rho-results}.  Thus the
projector and generalization experiments now differ only in the scalar error
functional plotted on the vertical axis.

The risk coefficient tail is nonzero at this branch point.  After removing
the generic \(k^{-1/2}\) factor, set \(h_k=\sqrt{k}\,g_{5,k}\) and fit
\begin{equation}
    h_k=a h_{k-1}+b h_{k-2},
    \qquad
    \widehat\rho_{\mathrm{gen},5}=\sqrt{-1/b}.
\end{equation}
The orders 800--1000 give the independent coefficient estimate
\(\widehat\rho_{\mathrm{gen},5}=1.24437676\).  This agrees with the algebraic
radius and confirms that the scalar trace does not cancel the closest
projector singularity.

The direct SVD comparison gives the following strict-series errors.
\begin{table}[htbp]
\centering
\begin{tabular}{ccccc}
\toprule
Noise level & Order 20 & Order 40 & Order 80 & Order 120\\
\midrule
\(0.8\rho_{\mathrm{gen},5}\) & \(1.62\times10^{-5}\) & \(1.91\times10^{-6}\)
 & \(1.17\times10^{-11}\) & \(1.33\times10^{-14}\)\\
\(1.2\rho_{\mathrm{gen},5}\) & \(4.92\times10^{-2}\) & \(2.62\times10^{1}\)
 & \(1.10\times10^{3}\) & \(2.49\times10^{7}\)\\
\bottomrule
\end{tabular}
\caption{Conditional population generalization-risk series errors below and
above the independently computed convergence radius.}
\label{tab:generalization-threshold}
\end{table}
At the value \(\sigma=2\) included in the notebook sweep, the exact
correction identity remains valid but the Taylor series must diverge because
\(2>\rho_{\mathrm{gen},5}\).

\begin{figure}[htbp]
    \centering
    \includegraphics[width=0.48\textwidth]
      {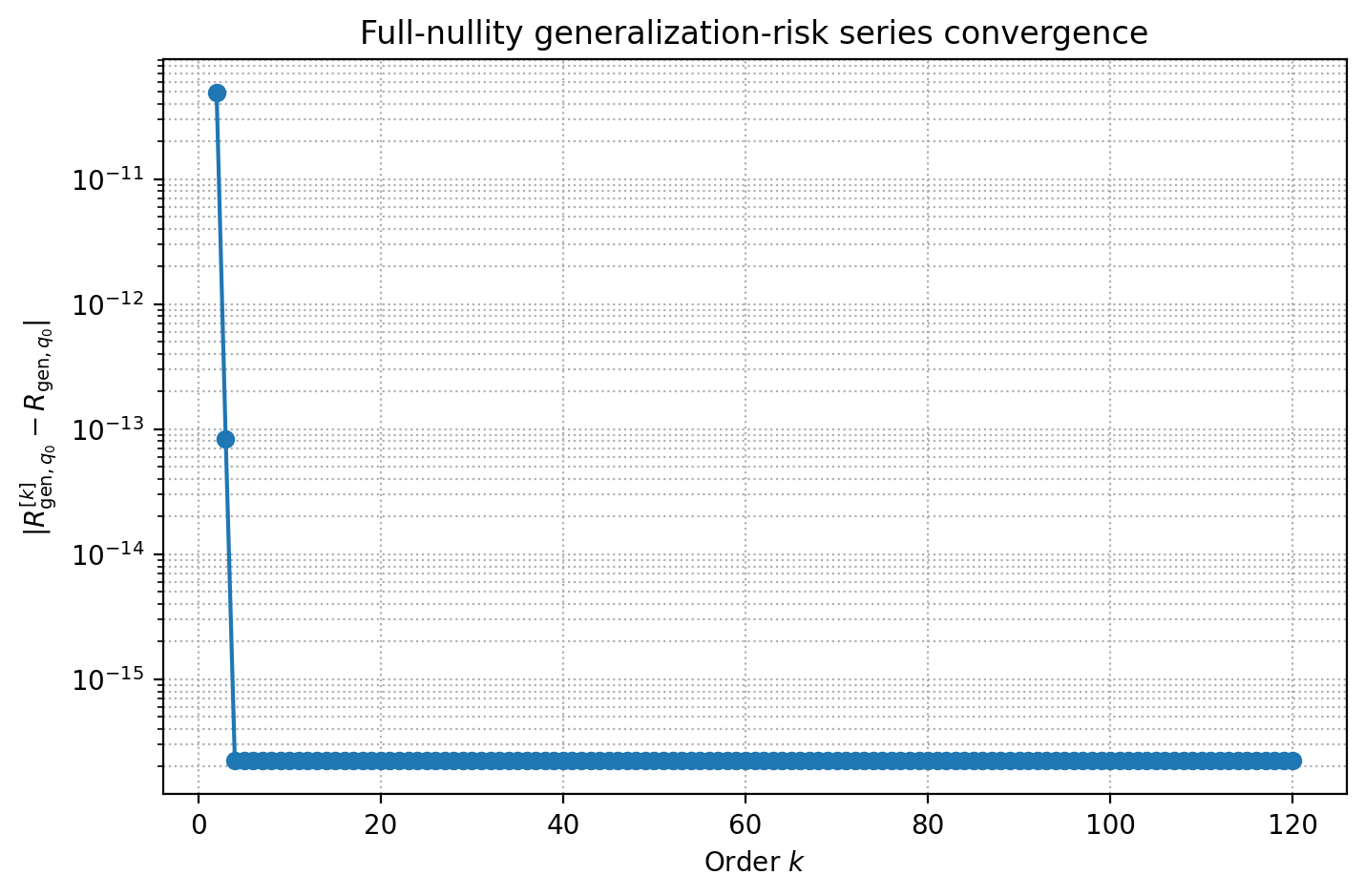}
    \hfill
    \includegraphics[width=0.48\textwidth]
      {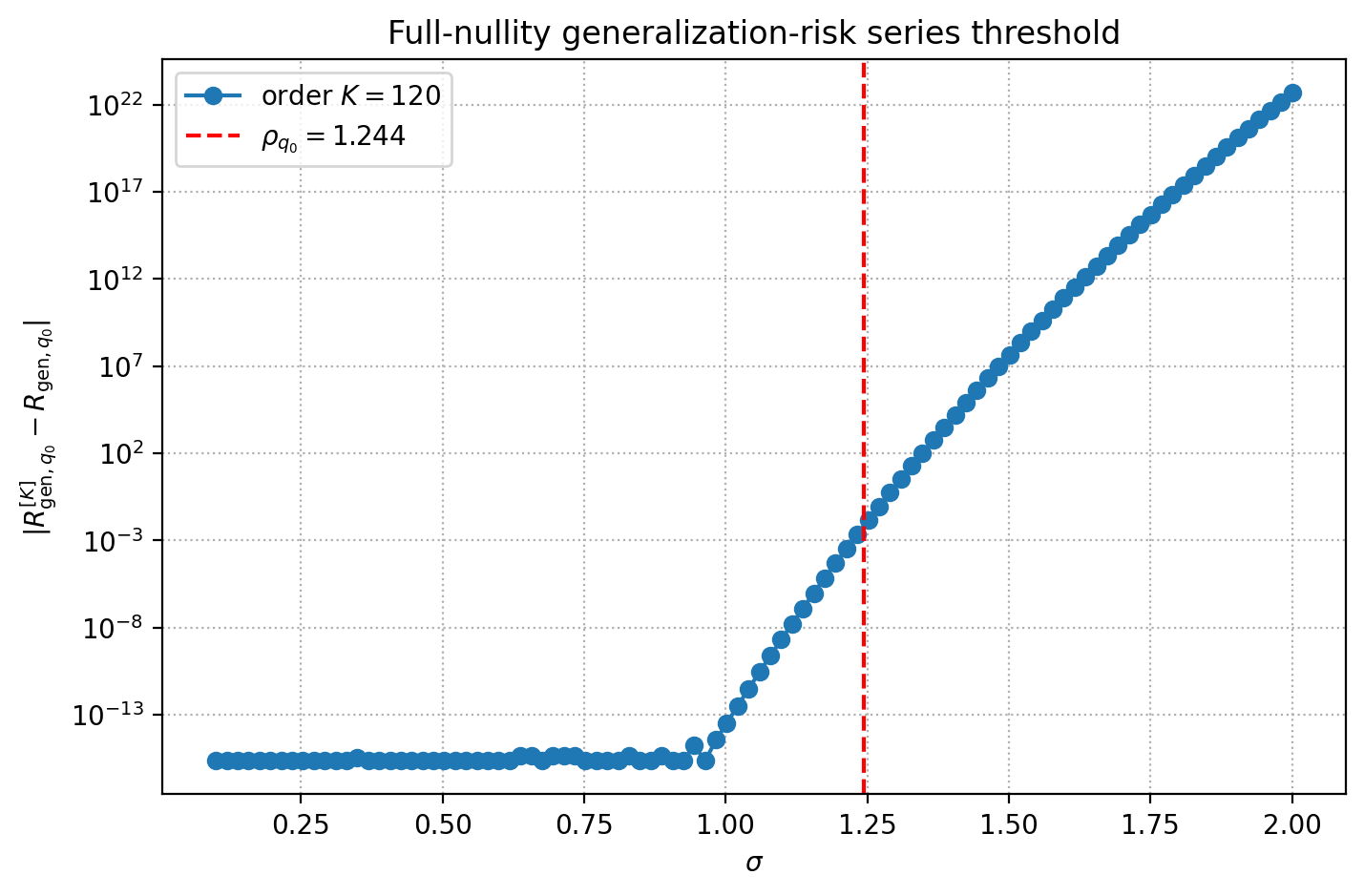}
    \caption{Full-nullity conditional generalization-risk series for the
    shared seed-8 realization.  Left: convergence in order at
    \(\sigma=10^{-3}\).  Right: the order-120 error
    \(\delta_{\mathrm{gen},q_0}^{[120]}\) over \(0.1\leq\sigma\leq2\),
    with the independently computed projector radius \(\rho_{q_0}\).}
    \label{fig:generalization-convergence-threshold}
\end{figure}

\begin{figure}[htbp]
    \centering
    \includegraphics[width=0.82\textwidth]
      {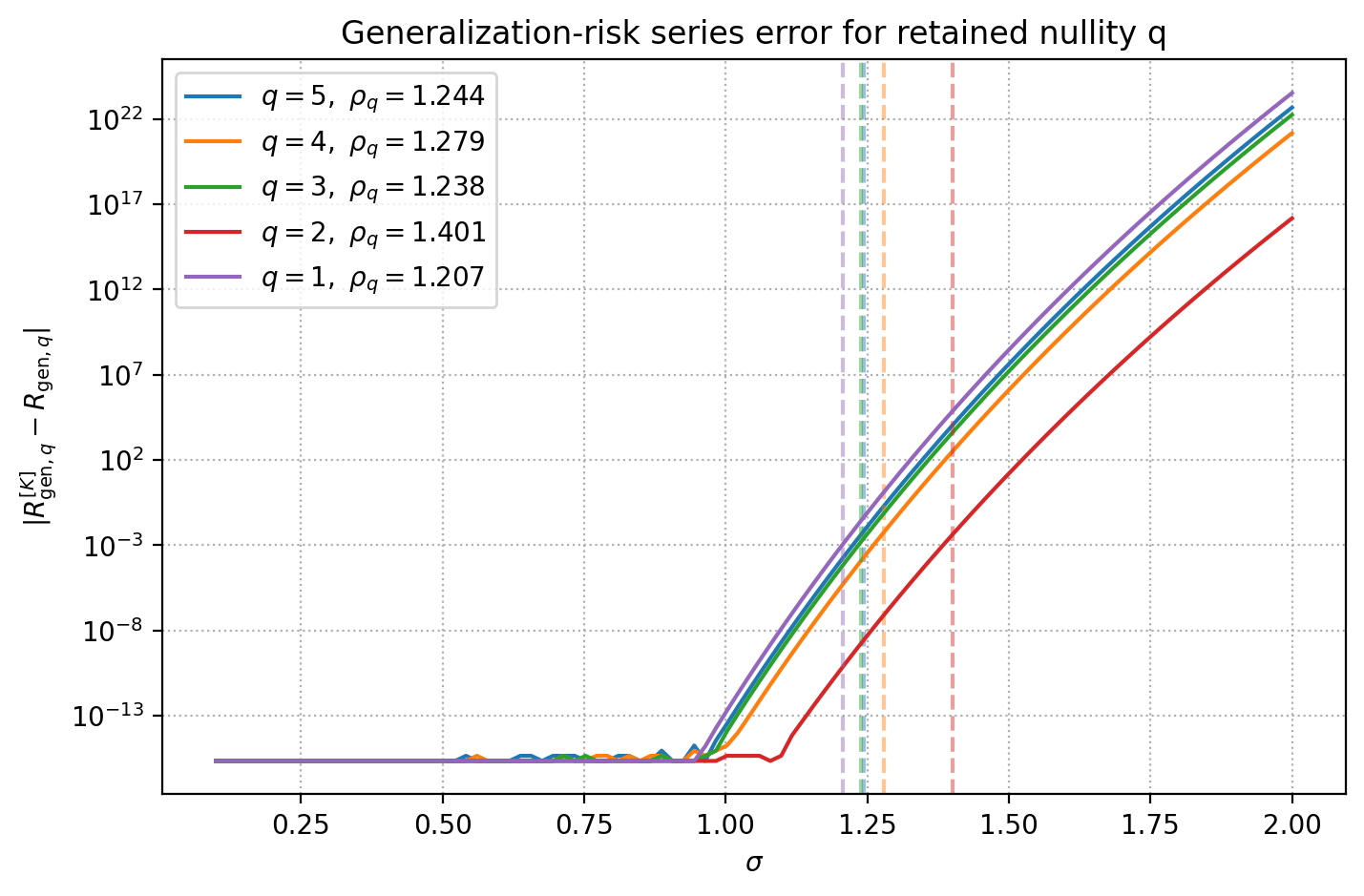}
    \caption{Order-120 conditional generalization-risk series error
    \(\delta_{\mathrm{gen},q}^{[120]}\) for each retained nullity.  The
    dashed lines are the independently computed projector radii \(\rho_q\).
    The observed growth at the same boundaries provides numerical evidence
    that, for this realization, no plotted scalar risk cancels its limiting
    projector singularity.}
    \label{fig:generalization-reduced-thresholds}
\end{figure}

\section{Reduced nullity}

Let
\begin{equation}
    W=\Phi\operatorname{diag}(\mu_0,\ldots,\mu_{q_0-1})\Phi^T,
    \qquad
    \mu_0\leq\cdots\leq\mu_{q_0-1},
    \label{eq:W-decomposition}
\end{equation}
where \(W\) is given by \eqref{eq:effective-splitting}, and write
\(\Phi_q=[\V{\phi}_0,\ldots,\V{\phi}_{q-1}]\).  For \(1\leq q<q_0\), if
\(\mu_{q-1}<\mu_q\), the bottom-\(q\) SVD subspace has the unique small-noise
limit
\begin{equation}
    Q_q=Q_0\Phi_q,
    \qquad
    \widehat P_q(\sigma)=Q_qQ_q^T+O(\sigma).
    \label{eq:reduced-limit}
\end{equation}
For \(q=q_0\), simply take \(\Phi_{q_0}=\Phi\) and
\(Q_{q_0}=Q_0\Phi\), which spans
the same complete null space as \(Q_0\).
The values \(\mu_j\) label the zero-origin branches because
\begin{equation}
    \lambda^{(j)}(\sigma)=\sigma^2\mu_j+O(\sigma^3).
\end{equation}

Reducing nullity from \(q\) to \(q-1\) moves one branch from the retained
cluster to its complement.  An old limiting collision can become internal,
but a previously internal collision can also become a new boundary
collision.  Consequently,
\begin{equation}
    \rho_{q-1}\ \text{may be greater or smaller than}\ \rho_q;
\end{equation}
there is no general monotonicity theorem.  Also, the full-nullity recursion
in \Cref{prop:block-recursion} cannot simply be called with an incomplete
clean null basis.  Reduced projectors must first be selected by
\eqref{eq:W-decomposition} and then continued as a spectral cluster.
Reducing \(q\) changes the estimated subspace; it extends the convergence
range of a different, lower-dimensional target and does not repair the
original full-nullity projector.

\section{Numerical verification}
\label{sec:numerics}

All values in this section are generated by the experiment notebooks and
their companion library functions.  Subspace-recursion errors are reported using
the Frobenius distance between orthogonal projectors, so they do not depend
on a rotation or sign chosen for the singular vectors.  Vector and scalar
risk errors are labelled separately.

\subsection{Compact formula}

We first test \Cref{prop:compact} on a problem satisfying its
one-dimensional-null-space assumption.  We use \(m=6\), \(\tau=50\),
\(\sigma=0.2\), random seed \(17\), and construct \(Z\) with
\(\operatorname{rank}(Z)=5\).  The supplied smallest noisy eigenvalue is
\begin{equation}
    \lambda_{\mathrm{SVD}}=1.76378905836.
\end{equation}
The comparison in \Cref{tab:compact-check} confirms both the compact formula
and the need for \(M=GEE^T\).

\begin{table}[htbp]
\centering
\small
\begin{tabular}{lcc}
\toprule
Definition of \(M\) & Vector difference from SVD & Eigenvector residual\\
\midrule
\(M=(ZZ^T)^+EE^T\) & \(4.51\times10^{-16}\) & \(1.47\times10^{-14}\)\\
\(M=(ZZ^T)EE^T\)   & \(3.28\times10^{-1}\)  & \(1.88\times10^{1}\)\\
\bottomrule
\end{tabular}
\caption{Verification of the exact compact reconstruction.  The correct
formula has zero normalization error to machine precision and
\(\operatorname{cond}(\mathcal{Q})=1.31077\).}
\label{tab:compact-check}
\end{table}

The notebook also performs an a-posteriori branchwise check in the
multiple-nullity example at \(\sigma=10^{-3}\).  It obtains
\begin{equation}
    \|\V{\epsilon}_{\mathrm{SVD}}\|_2
      =3.33314836\times10^{-5},
    \quad
    \|\V{\epsilon}_{\mathrm{compact}}
       -\V{\epsilon}_{\mathrm{SVD}}\|_2
      =2.19\times10^{-16},
\end{equation}
with eigenvector residual \(9.12\times10^{-15}\).  Because the clean vector
in that check is matched to the noisy branch after the SVD, the independent
simple-nullity test above is the direct validation of the proposition.

\subsection{Infinite-order convergence}

For the multiple-nullity experiment, we use
\begin{equation}
    m=6,\qquad n=1,\qquad \tau=100,
\end{equation}
and random seed \(8\).  Draw \(H_0\), \(X\), and \(\cE\) with independent
standard-normal entries, and set \(H=H_0/\|H_0\|_F\).  Thus \(Z=HX\),
\(E(\sigma)=\sigma\cE\), \(\widetilde Z=Z+E(\sigma)\), and \(q_0=5\).

At \(\sigma=10^{-3}\), the projector-series errors for recursion orders
\(1,2,3,4\), and \(120\) are
{\small
\begin{equation}
\begin{array}{c|ccccc}
K&1&2&3&4&120\\ \hline
\|P^{[K]}_5-\widehat P_5\|_F
&3.7381\!\times\!10^{-7}
&1.0869\!\times\!10^{-10}
&1.2732\!\times\!10^{-13}
&1.4551\!\times\!10^{-15}
&1.4536\!\times\!10^{-15}
\end{array}
\end{equation}
}
so four orders reach floating-point accuracy.

\begin{figure}[htbp]
    \centering
    \includegraphics[width=0.72\textwidth]{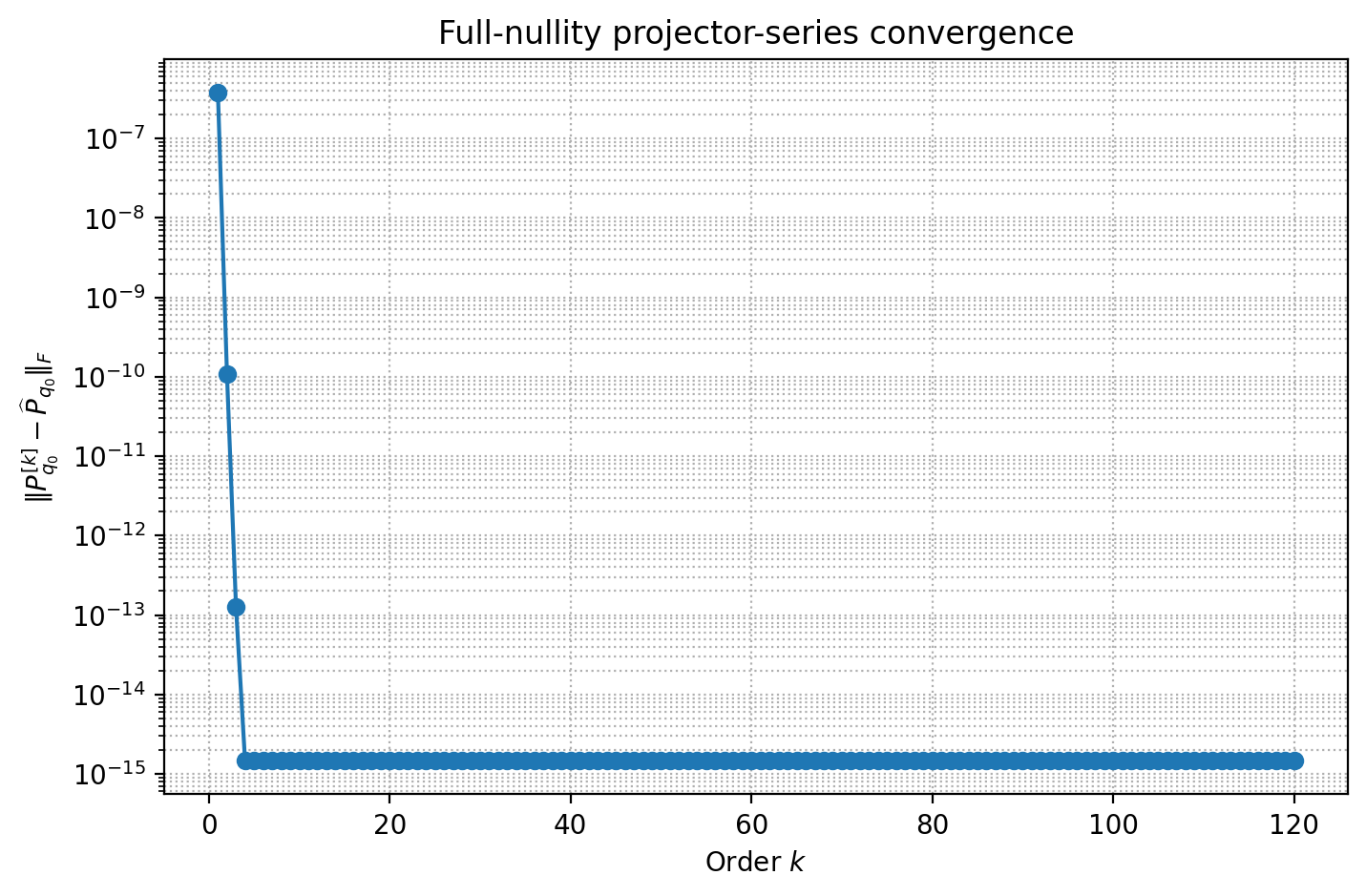}
    \caption{Convergence of the block recursion to the exact full-nullity SVD
    projector at \(\sigma=10^{-3}\).}
    \label{fig:order-convergence}
\end{figure}

\subsection{Ordered-branch generalization regimes}

This ensemble experiment is separate from the fixed seed-8 study of
convergence radii.  Set
\begin{equation}
 m=7,\qquad n=3,\qquad q_0=4,\qquad \tau=100,
 \label{eq:rank-experiment-size}
\end{equation}
draw \(H_0\) with seed 8, and use
\begin{equation}
 H=\frac{2H_0}{\|H_0\|_F},
 \qquad h_\star=\|H\|_2=1.57970420.
 \label{eq:rank-experiment-H}
\end{equation}
For each noise value, the notebook averages over $100{,}000$ independent
training realizations.  A calibration run with seed 73119 uses
$0.15\leq\zeta\leq0.45$ and $5\leq\zeta\leq10$; an independent
validation run with seed 20260803 uses 121 logarithmically spaced values
over $10^{-2}\leq\zeta\leq10$.  Both risk profiles are normalized only
after their Monte Carlo expectations are computed.

\begin{figure}[htbp]
 \centering
 \includegraphics[width=0.48\textwidth]
   {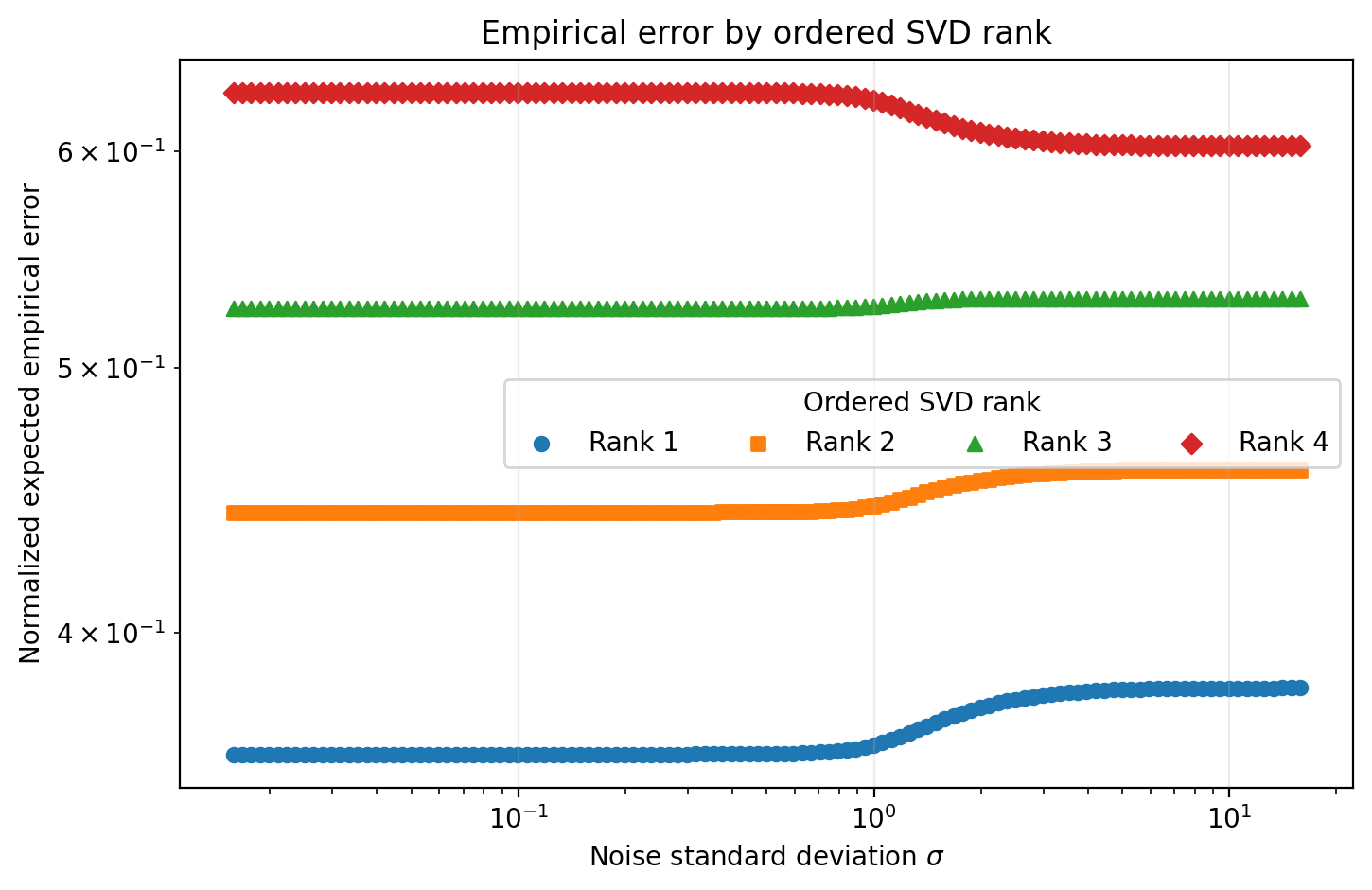}
 \hfill
 \includegraphics[width=0.48\textwidth]
   {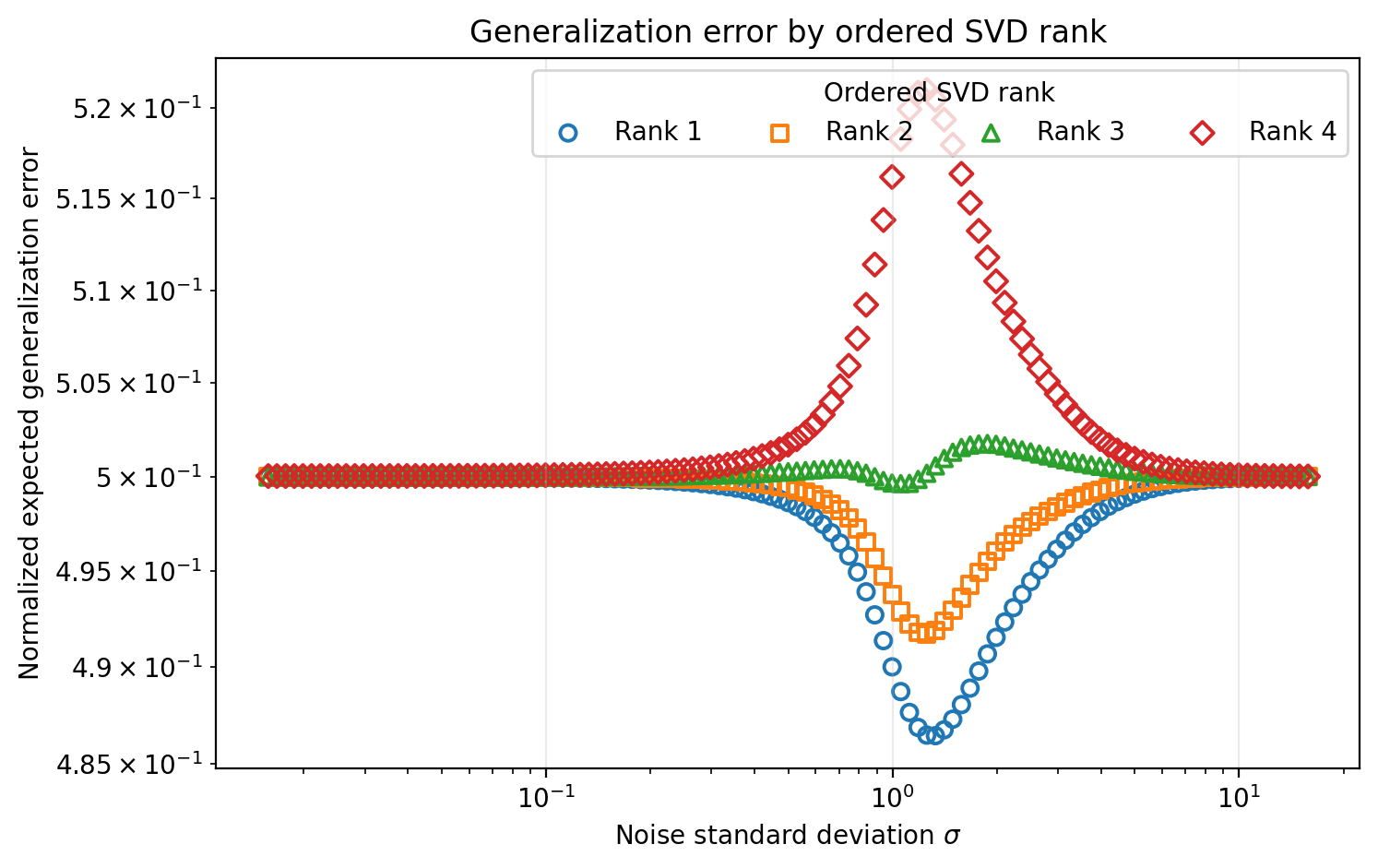}
 \caption{Normalized expected branch risks.  Left: empirical risks are
 ordered exactly, and their second-order separation is governed by \(W\) as
 proved in \Cref{prop:empirical-ranking}.  Right: generalization
 risks are nearly equal at the two noise endpoints and visibly separated at
 intermediate noise.  The marker label ``Rank \(i+1\)'' denotes the ordered
 SVD rank, not an intrinsic dimension of a null vector.}
 \label{fig:ordered-branch-risks}
\end{figure}

The signal covariance in this experiment is not isotropic: its nonzero
eigenvalues, in increasing order, are
\begin{equation}
 (\beta_5,\beta_6,\beta_7)
  =(0.702381824,\ 0.802152806,\ 2.495465370).
 \label{eq:rank-experiment-population-eigenvalues}
\end{equation}
We therefore checked the unequal-spike mechanism in
\Cref{prop:overlap-ranking} directly.  A dedicated run with seed 20260803,
\(3{,}000{,}000\) independent trials, and \(\zeta=1\) gave the following
Monte Carlo estimates of the structural risks and cumulative
population-overlap tails.  The last three columns are respectively the
overlap with all three signal modes, the top two modes, and the top mode.
\begin{table}[H]
\centering
\small
\begin{tabular}{ccccc}
\toprule
SVD rank & \(\widehat{\mathbb E}[\widehat{\V{u}}_{(i)}^TB
 \widehat{\V{u}}_{(i)}]\)
 & \(\sum_{k=5}^7\widehat\Omega_{ik}\)
 & \(\sum_{k=6}^7\widehat\Omega_{ik}\)
 & \(\widehat\Omega_{i7}\)\\
\midrule
1 & 0.131298 & 0.13838 & 0.07337 & 0.01582\\
2 & 0.160640 & 0.17201 & 0.09029 & 0.01820\\
3 & 0.204187 & 0.22247 & 0.11585 & 0.02148\\
4 & 0.283649 & 0.31508 & 0.16321 & 0.02720\\
\bottomrule
\end{tabular}
\caption{Monte Carlo estimates of signal leakage and overlap tails at the
representative middle-noise point \(\sigma=\|H\|_2\).  Every estimated tail
increases with ordered sample rank; the simultaneous certificate below and
\eqref{eq:overlap-gap-identity} then establish the expected-risk ordering.}
\label{tab:middle-overlap-certificate}
\end{table}

The sampled adjacent structural-risk gaps are
\begin{equation}
 (0.029342,\ 0.043548,\ 0.079461).
 \label{eq:middle-risk-gaps}
\end{equation}
This conclusion is not driven by a few pointwise ordered samples: only
about \(10.1\%\) in a separate \(200{,}000\)-trial check had all four
structural risks ordered.  The order therefore appears after expectation,
not pointwise.  For a distribution-free certificate of the overlap
criterion, let \(D_{t,i,\ell}\) be trial \(t\)'s difference between ranks
\(i+1\) and \(i\) in one of the three tail-overlap columns of
\Cref{tab:middle-overlap-certificate}.  There are \(K=9\) adjacent tail
margins and each \(D_{t,i,\ell}\in[-1,1]\).  One-sided Hoeffding bounds and a
union bound give, with probability at least \(1-\alpha\),
\begin{equation}
 \mathbb E D_{i,\ell}
 \geq \widehat D_{i,\ell}
  -\sqrt{\frac{2}{N}\log\!\frac{K}{\alpha}}
 \quad\text{simultaneously for all }(i,\ell).
 \label{eq:middle-Hoeffding-certificate}
\end{equation}
For \(N=3{,}000{,}000\) and \(\alpha=0.01\), the subtraction radius is
0.002130.  The smallest sampled tail margin is 0.002382, leaving a positive
simultaneous lower endpoint 0.000252.  Thus, with at least \(99\%\)
confidence, all nine population tail inequalities hold.
\Cref{prop:overlap-ranking} then proves all three positive expected-risk gaps
at this representative middle-noise point; the direct gaps are reported in
\eqref{eq:middle-risk-gaps}.

The coefficients calibrated from the raw expected generalization risks are
\begin{equation}
 C_2=0.024350,\qquad C_4=0.041545,\qquad C_H=0.226569.
 \label{eq:rank-fitted-coefficients}
\end{equation}
For \(\delta_{\mathrm{vis}}=10^{-3}\), the independent validation curve gives
\begin{equation}
 [\sigma_{\mathrm{lo}},\sigma_{\mathrm{hi}}]_{\mathrm{measured}}
   =[0.3102,\,6.0161].
 \label{eq:rank-measured-interval}
\end{equation}
The older finite-window linear fit used by the notebook gives 0.2363.  This
coefficient is not the asymptotic \(C_2\) in
\eqref{eq:rank-fitted-coefficients}: using that \(C_2\) alone gives 0.3201.
For this experiment \(C_4>0\), and including it moves the lower prediction
to 0.3101.  Together with the high-noise formula this gives
\begin{equation}
 [\sigma_{\mathrm{lo}},\sigma_{\mathrm{hi}}]_{\mathrm{predicted}}
   =[0.3101,\,6.1288],
 \label{eq:rank-predicted-interval}
\end{equation}
with relative endpoint errors $0.04\%$ and $1.87\%$, respectively.

\begin{figure}[htbp]
 \centering
 \includegraphics[width=0.76\textwidth]
   {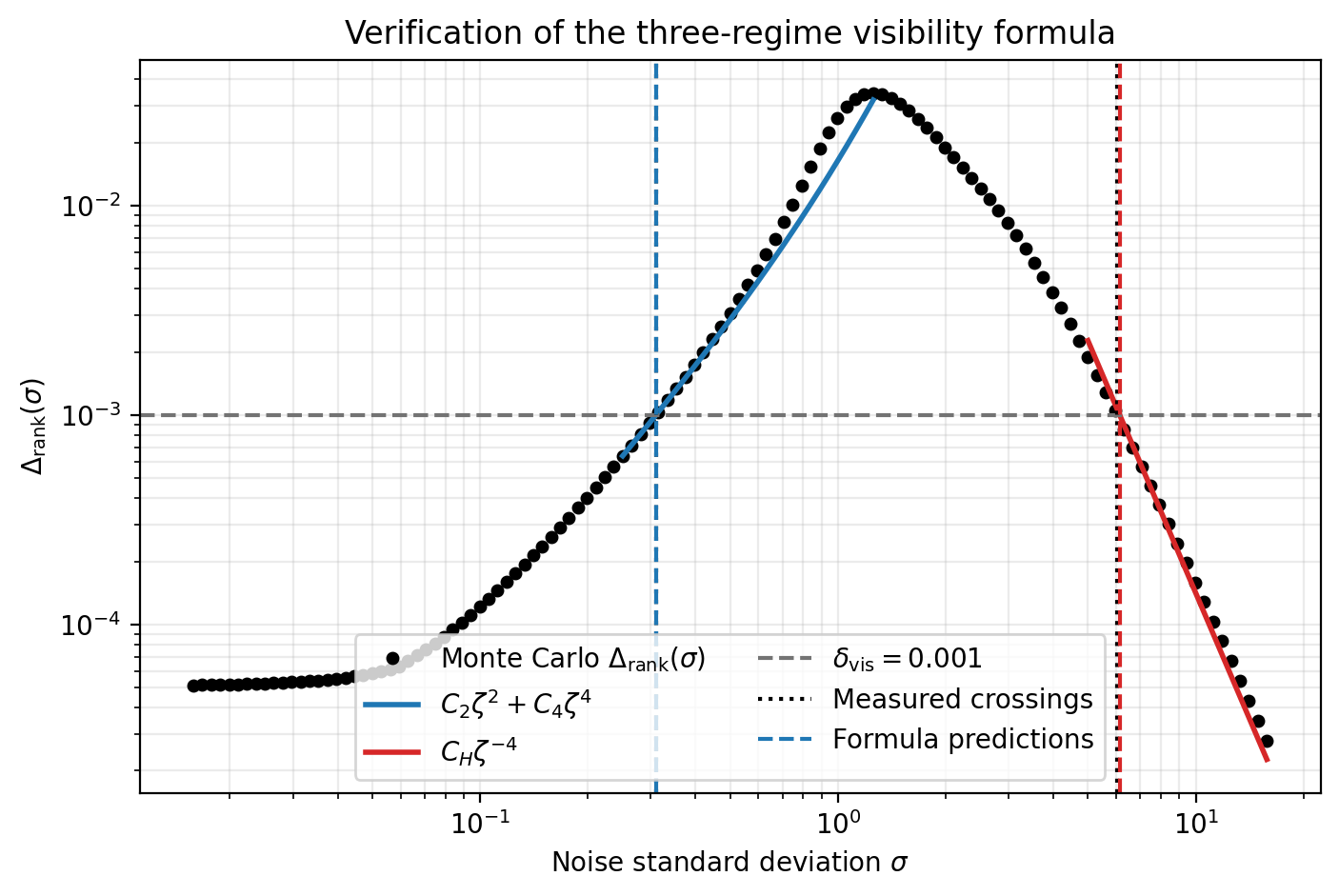}
 \caption{Monte Carlo rank spread \(\Delta_{\mathrm{rank}}\), its calibrated
 low-noise law \(C_2\zeta^2+C_4\zeta^4\), and its calibrated high-noise law
 \(C_H\zeta^{-4}\).  Vertical lines compare the measured and predicted
 crossings of \(\delta_{\mathrm{vis}}=10^{-3}\).}
 \label{fig:rank-threshold-verification}
\end{figure}

The calibration validates the asymptotic powers and the corrected crossing
formula out of sample; it is not a parameter-free derivation of
\(C_2,C_4,C_H\).  At the smallest finite plotted noise there is a genuine
population separation of order \(C_2\zeta^2\), superimposed on Monte Carlo
and floating-point error; \eqref{eq:rank-small-spread} forces only its
\(\zeta\downarrow0\) limit to be zero.
For an isotropic signal subspace the rank order is exact by
\Cref{thm:isotropic-middle-ranking}; for this unequal-spike matrix the overlap
tails explain the order and \eqref{eq:middle-Hoeffding-certificate} certifies
it at the representative middle-noise point.

\subsection{Exceptional-point and reduced-nullity results}

For the same realization, the second-order lifted eigenvalues are
\begin{equation}
    (\mu_0,\ldots,\mu_4)
    =(72.2180,\ 78.9868,\ 102.1430,\ 113.0325,\ 146.6655).
    \label{eq:numerical-W}
\end{equation}
They are distinct, so every reduced small-noise limit in
\eqref{eq:reduced-limit} is unambiguous.  The roots of
\eqref{eq:discriminant} were computed by a rank-preserving rationalization of
the compact SVD factors to eight digits, followed by numerical root finding
and branch continuation from the \(W\)-labels.  The reported values are
therefore algebraic--numerical estimates for the floating-point realization.
One member of each conjugate pair is reported in \Cref{tab:rho-results};
labels \(0,\ldots,4\) originate in the clean null space, while label \(5\) is
the clean positive-eigenvalue branch.

\begin{table}[htbp]
\centering
\small
\begin{tabular}{ccccc}
\toprule
Retained nullity \(q\) & \(\Re s_{\star,q}\) & \(\Im s_{\star,q}\)
& \(\rho_q=|s_{\star,q}|\) & Collision\\
\midrule
5 & \(-1.16405872\) & \(-0.43982812\) & \(1.24437996\) & \((4,5)\)\\
4 & \(-1.08848353\) & \(-0.67174264\) & \(1.27907567\) & \((3,5)\)\\
3 & \(-1.21293931\) & \(-0.24914546\) & \(1.23826299\) & \((2,3)\)\\
2 & \(-0.97696847\) & \(-1.00453024\) & \(1.40126671\) & \((0,5)\)\\
1 & \(-0.25176277\) & \(-1.18001336\) & \(1.20657201\) & \((0,1)\)\\
\bottomrule
\end{tabular}
\caption{Exceptional points crossing the boundary between the retained
cluster and its complement.  The thresholds are computed from \(Z\) and
\(\cE\), not fitted from the projector-series-error curves.}
\label{tab:rho-results}
\end{table}

At \(\sigma=1.26\), which lies between \(\rho_5\) and \(\rho_4\), the
full-nullity errors at orders \(20,40,80,120\) are
\begin{equation}
    (0.0546,\ 0.0601,\ 0.0558,\ 0.0710),
\end{equation}
while the retained-nullity-four errors are
\begin{equation}
    (0.0346,\ 0.0192,\ 0.00754,\ 0.00342).
\end{equation}
Thus the \(q=4\) series converges at this noise level whereas the original
\(q=5\) Taylor series does not.

For each \(q\), Taylor coefficients of the selected projector were recovered
by a \(1024\)-point Cauchy--Fourier calculation on the circle
\(|s|=0.95\rho_q\).  The order-120 error is
\begin{equation}
    d_q^{[120]}(\sigma)
      =\|P_q^{[120]}(\sigma)-\widehat P_q(\sigma)\|_F.
\end{equation}

\begin{figure}[htbp]
    \centering
    \includegraphics[width=0.82\textwidth]{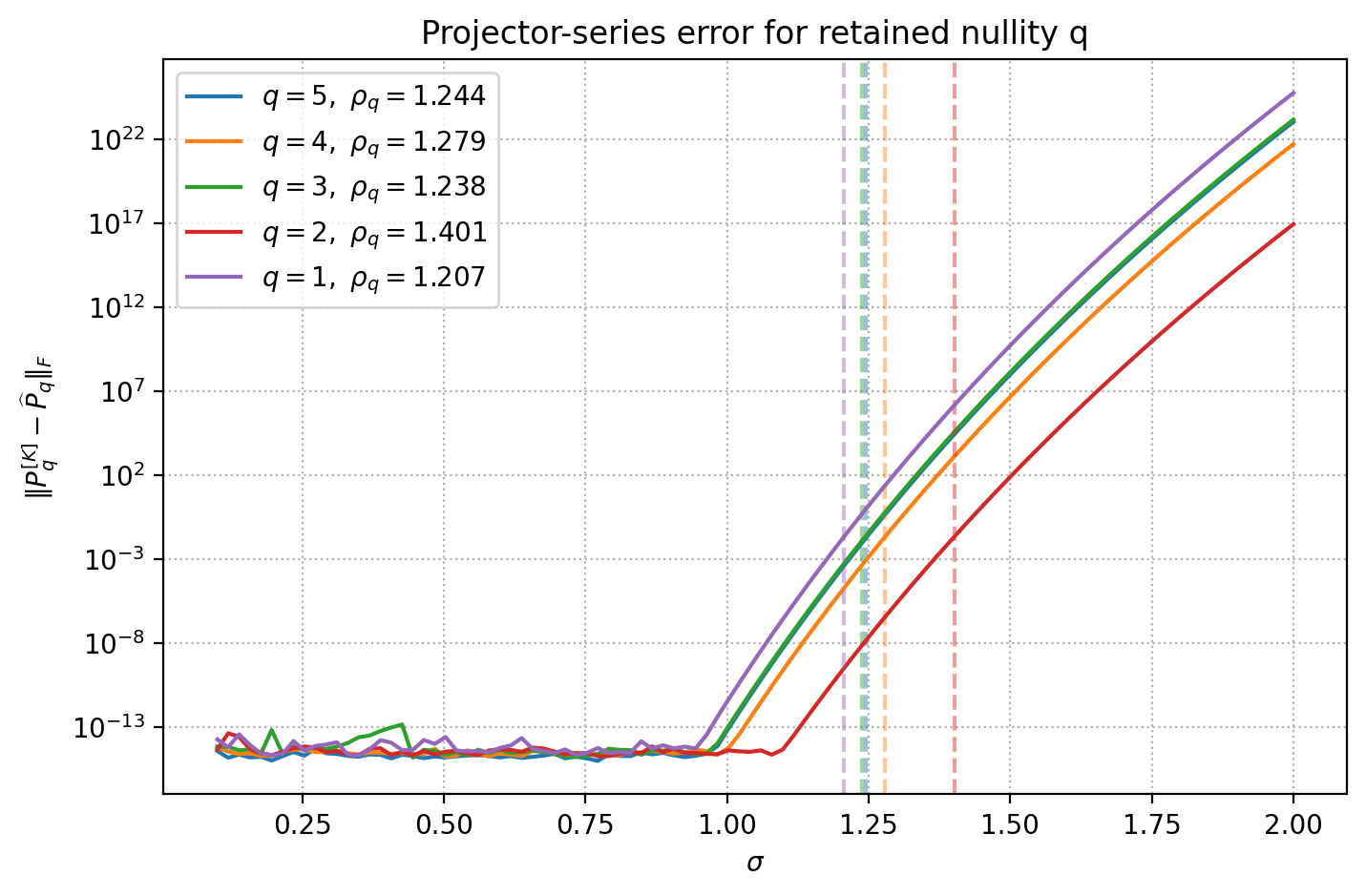}
    \caption{Order-120 projector-series error for the \(q\) smallest left
    singular-vector subspaces.  Dashed lines are the independently computed
    radii \(\rho_q\).  Moving the boundary between the retained cluster and
    its complement changes the collision that limits convergence.}
    \label{fig:reduced-thresholds}
\end{figure}

\begin{table}[htbp]
\centering
\small
\setlength{\tabcolsep}{4pt}
\resizebox{\textwidth}{!}{%
\begin{tabular}{cccccc}
\toprule
\(\sigma\) & \(q=5\) & \(q=4\) & \(q=3\) & \(q=2\) & \(q=1\)\\
\midrule
1.00
&\(5.56\times10^{-14}\)&\(6.37\times10^{-15}\)&\(7.87\times10^{-14}\)
&\(3.24\times10^{-15}\)&\(3.06\times10^{-12}\)\\
1.20
&\(1.99\times10^{-4}\)&\(9.46\times10^{-6}\)&\(2.72\times10^{-4}\)
&\(1.58\times10^{-10}\)&\(1.07\times10^{-2}\)\\
1.26
&\(7.10\times10^{-2}\)&\(3.39\times10^{-3}\)&\(9.74\times10^{-2}\)
&\(5.68\times10^{-8}\)&\(3.84\times10^{0}\)\\
1.35
&\(2.89\times10^{2}\)&\(1.38\times10^{1}\)&\(3.97\times10^{2}\)
&\(2.32\times10^{-4}\)&\(1.56\times10^{4}\)\\
1.40
&\(2.31\times10^{4}\)&\(1.10\times10^{3}\)&\(3.17\times10^{4}\)
&\(1.86\times10^{-2}\)&\(1.25\times10^{6}\)\\
1.50
&\(9.42\times10^{7}\)&\(4.49\times10^{6}\)&\(1.29\times10^{8}\)
&\(7.57\times10^{1}\)&\(5.06\times10^{9}\)\\
\bottomrule
\end{tabular}}
\caption{Order-120 Frobenius projector-series errors.  Each series is accurate
inside its corresponding radius and grows rapidly after that radius is
crossed.}
\label{tab:order120-errors}
\end{table}

The results make four points precise.
\begin{enumerate}
    \item The full-nullity threshold is
    \(\rho_5=1.24437996\), corresponding to
    \(s_\star=-1.16405872-0.43982812\,\mathrm{i}\).
    \item Reducing nullity from \(5\) to \(4\) increases the radius to
    \(1.27907567\), and \(q=2\) gives the largest radius,
    \(\rho_2=1.40126671\).
    \item The improvement is nonmonotone: \(q=3\) and \(q=1\) expose closer
    boundary collisions.
    \item The value \(\rho_2=1.40126671\) is not contradicted by the closer
    discriminant roots.  At \(1.206572\) both colliding labels \((0,1)\) are
    retained by \(q=2\); the candidates at \(1.238263\), \(1.244380\), and
    \(1.279076\) are internal to its complement.  The first collision crossing
    the \(q=2\) boundary is \((0,5)\).
\end{enumerate}

\section{Relation to the BBP transition}

The Baik--Ben Arous--P\'ech\'e transition describes the asymptotic separation
of a population spike from a random-matrix spectral bulk
\cite{bbp2005}.  For the standard additive rectangular model
\begin{equation}
    Y=X_0+\frac{\sigma}{\sqrt{\tau}}G,
    \qquad \frac{m}{\tau}\longrightarrow c,
\end{equation}
a singular-value spike of strength \(\theta\) separates from the noise bulk
under the following condition \cite{benaych2012}:
\begin{equation}
    \frac{\theta}{\sigma}>c^{1/4}
    \qquad\Longleftrightarrow\qquad
    \sigma<\sigma_{\mathrm{BBP}}=\frac{\theta}{c^{1/4}}
    \label{eq:rectangular-bbp}
\end{equation}

The reduced-nullity experiment and BBP theory share the same structural
idea: a selected spectral branch is reliable while it is separated from its
complement, and changing the selected rank changes which branch crossing is
relevant.  With several spikes, this naturally leads to a sequence of rank
transitions.

The threshold notions are nevertheless different by definition.  The
quantity \(\rho_q\) in \eqref{eq:rho-formula} is finite-dimensional, depends
on the realized matrices \(Z\) and \(\cE\), and is the complex-plane Taylor
radius along a fixed noise direction.  The visibility endpoints
\(\sigma_{\mathrm{lo}}(\delta_{\mathrm{vis}})\) and
\(\sigma_{\mathrm{hi}}(\delta_{\mathrm{vis}})\) instead come from an
ensemble-averaged normalized generalization statistic and change with the
chosen tolerance.  The BBP threshold is a deterministic real-axis phase
transition obtained after dimensions grow at a fixed aspect ratio.

The middle visibility interval is consistent with signal--noise mixing in a
finite sample.  BBP theory also depends on the boundary of the selected rank.
The visibility endpoints, however, depend on the tolerance and are not phase
transitions.

Establishing an asymptotic connection would require a suitable statistic---for
example \(\rho_q\) or the location of the maximum rank spread---and a proof,
over increasing \((m,\tau)\) and many noise realizations, that
\begin{equation}
    \rho_q^{(m,\tau)}\longrightarrow\sigma_{\mathrm{BBP}}
\end{equation}
and that the singular-vector overlap changes at the same limit.  The present
experiments therefore provide finite-sample analytic mechanisms, or
precursors, for BBP-type rank transitions; they do not yet prove an
asymptotic identity.

The present realization also demonstrates the distinction numerically.  If
we insert it naively into the asymptotic normalization, then
\begin{equation}
    c=\frac{6}{100}=0.06,
    \qquad
    \widehat\theta_Z=\frac{s_1(Z)}{\sqrt{100}}=1.0408844,
    \qquad
    \sigma_{\mathrm{BBP}}
       =\frac{\widehat\theta_Z}{c^{1/4}}\approx2.10312.
\end{equation}
This is not equal to the finite-sample full-nullity radius
\(\rho_5\approx1.24438\), as expected for one small realization outside a
dimension-growing experiment.

\section{Conclusion}

The compact formula, infinite-order recursion, and exceptional-point analysis
describe different aspects of the same SVD estimator.  For simple nullity,
the compact representation is exact, with the noisy eigenvalue determined by
one scalar closure equation, and its correct quadratic term is
\(M=(ZZ^T)^+EE^T\).  The first-order error is the first term of the Taylor
expansion of this closed compact system.  The conditional
population generalization risk is the exact quadratic functional
\(\sigma^2+\V{\epsilon}^TB\V{\epsilon}\), so it inherits both a compact
form and a consistently truncated coefficient convolution.  For multiple
nullity, the complete basis must be propagated as a block, and the
second-order matrix \(W\) selects the small-noise branches used by a reduced
SVD target.  
These identities describe the fixed-realization risks
\(R_{\mathrm{emp},q}\) and \(R_{\mathrm{gen},q}\).
The template's expected
errors \(\varepsilon_{\mathrm{emp},q}\) and
\(\varepsilon_{\mathrm{gen},q}\) are their outer training expectations;
their coefficient series follows termwise only when expectation and the
Taylor expansion can be interchanged.

For ordered SVD branches under Gaussian training, the low-noise
analysis proves the common leading law
\((1+n/(\tau-n-1))\sigma^2\); branch dependence first enters through the
fourth-order projector coefficients supplied by the infinite-order theory.
A separate inverse-noise expansion proves a second common endpoint,
\(\sigma^2+\|H\|_F^2/m\), with branch differences of order
\(\sigma^{-2}\).  Hence the normalized expected generalization profile is
asymptotically uniform at both endpoints.  In contrast, empirical branch
risks are exactly ordered at every noise level, and their normalized
small-noise profile converges to the strictly ordered mean-eigenvalue profile
of \(W\sim\operatorname{Wishart}_{q_0}(\tau-n,I)\).

The middle ordering is also now separated into theorem and experiment.  If
all nonzero eigenvalues of \(HH^T\) are equal, a conditional Wishart
column-swap argument proves strict expected generalization order at every
finite positive noise level; the endpoint expansions explain why this order
is visible principally at intermediate noise.  For unequal spikes, the exact
risk-gap identity weights cumulative population-overlap differences by the
population spectral gaps.  All relevant overlaps increase with SVD rank in
the seed-8 experiment, and a distribution-free concentration bound certifies
the three positive adjacent expected gaps at \(\sigma=\|H\|_2\) with
simultaneous \(99\%\) Monte Carlo confidence.  The sixth-order low-noise term
improves the lower-crossover estimate for this experiment, while the
high-noise spread is \(O((h_\star/\sigma)^4)\).

The exact convergence radius is the nearest exceptional point that crosses
the boundary between the retained cluster and its complement.  It is
computed from the discriminant and branch labels, independently of the
observed error curve.  A scalar generalization risk generically has the same
radius; only an exact cancellation of the exceptional-point monodromy can
increase it.  In the shared seed-8 experiment no such cancellation occurs,
and \(\rho_{\mathrm{gen},5}=1.24437996\).  The reduced generalization plot uses
the same nonmonotone radii as the projector plot, with \(q=2\) largest at
\(1.40126671\).  Reducing nullity can remove the collision that limits one
projector and thereby create a larger convergence disk, but the
thresholds are not monotone because moving the boundary can expose a
different, closer collision.  In the reproducible experiment, \(q=2\) is
optimal among the five tested nullities, with \(\rho_2=1.40126671\).

Finally, the divergence shown in the figures is not a numerical failure of
the exact SVD.  It is the failure of a Taylor expansion evaluated outside its
disk of convergence.  The analysis explains when a perturbative SVD error
formula stops representing the exact SVD subspace, and it identifies the
finite-dimensional spectral-separation mechanism related to BBP-type rank
transitions.  The visibility interval, the exceptional-point radius, and the
asymptotic BBP threshold are related through spectral mixing but remain
mathematically distinct quantities.

\end{document}